\documentclass[12pt]{article}
\usepackage{}
\usepackage{mathrsfs}
\usepackage{amsfonts}
\usepackage{amssymb,amsmath}
\usepackage{cite}
\usepackage{cases}
\usepackage{amsthm}
\usepackage{CJK}

\allowdisplaybreaks
\def\sc{\scriptstyle}

\def\UU{{\frak U}}
\def\SS{{\cal S}}
\def\Der{{\rm Der}}
\def\Inn{{\rm Inn}}
\def\Ker{{\rm Ker}}
\def\Im{{\rm Im}}
\def\cl{\centerline}
\def\ll{\leftline}
\def\rar{\rightarrow}

\def\vs{\vspace*}
\def\ni{\noindent}
\def\VV{\mathcal {V}}
\def\Z{\mathbb{Z}}
\def\C{\mathbb{C}}
\def\QED{\hfill$\Box$}
\def\d{\lambda}
\def\e{\mu}
\def\f{\nu}
\def\g{\eta}
\def\h{\xi}
\def\G{\zeta}
\def\v{\gamma}
\def\BZ{\mathbf{Z}}

\def\D{\Delta}
\def\LL{{\cal L}}
\def\sl{\frak{sl}_2(\mathbb{C}_q)}

\def\wsl{\mbox{\footnotesize$\widetilde{\frak{sl}_2(\mathbb{C}_q)}$}}

\def\Cq{\cq[x^{\pm1},y^{\pm1}]}
\def\sm{\!\setminus\!}

\def\cq{{\cal C}_q}
\def\bz{{\bf 0}}
\def\bk{{\bf k}}

\def\bm{{\bf m}}
\def\DD{{d}}
\def\dir{\mathcal {D}}
\def\bn{{\bf n}}

\def\EE{e}
\def\FF{f}
\def\GG{g}
\def\HH{h}
\def\WL{\widetilde{\cal{L}}}

\def\em{e_{\bf m}}

\def\fm{f_{\bf m}}

\numberwithin{equation}{section}
\newtheorem{theo}{Theorem}[section]
\newtheorem{defi}[theo]{Definition}
\newtheorem{coro}[theo]{Corollary}
\newtheorem{lemm}[theo]{Lemma}
\newtheorem{prop}[theo]{Proposition}
\newtheorem{rema}[theo]{Remark}
\newtheorem{clai}{Claim}

\begin{document}
\begin{CJK*}{GBK}{song}
\cl{{\large\bf Lie bialgebra structures on the extended affine
Lie algebra $\wsl$
\,$^*$}\footnote {$\!\!\!\!\!\!\!^*\,$Supported by NSF grants (No 10825101, 11101056) and the China Postdoctoral Science Foundation Grant (No 201003326)\\
\indent}}\vs{6pt}

\cl{Ying Xu$^{a,b}$, Junbo Li$^{c}$}\vs{6pt}

\cl{\footnotesize $^{a}$ School of Mathematics of Hefei University of Technology, Hefei 230009, China}

\cl{\footnotesize $^{b}$ School of Mathematical Sciences, University of Science and Technology of China, Hefei 230026, China}

\cl{\footnotesize $^{c}$ School of Mathematics and Statistics, Changshu Institute of Technology, Changshu 215500, China}

\cl{\footnotesize Email: xying@mail.ustc.edu.cn, sd\_junbo@163.com}
\vs{8pt}

{\small
\parskip .005 truein
\baselineskip 3pt \lineskip 3pt

\noindent{{\bf Abstract.}\ \ 
Lie bialgebra structures on the extended affine Lie algebra $\wsl$
are investigated.
In particular, all Lie bialgebra structures on $\wsl$
are shown to be triangular coboundary. This result is obtained by employing some techniques, which may also work for more general extended affine Lie algebras, to prove the triviality of
the first cohomology group of $\wsl$ with coefficients in the tensor product of its adjoint module, namely, $H^1(\wsl,\wsl\otimes\wsl)=0$.
\vs{2pt}

}\vs{5pt}

}
\parskip .001 truein\baselineskip 6pt \lineskip 6pt
\vs{10 pt}

\ll{\bf\large{1\ \ \
Introduction}}\setcounter{section}{1}\setcounter{equation}{0}\setcounter{theo}{0}\vs{4 pt}

\noindent It is well known that extended affine Lie algebras (EALAs), first introduced in \cite{HT} in the sense of quasi-simple Lie algebras and systematically investigated in \cite{AABGP}, can be regarded as  higher dimensional generalizations of affine Lie algebras. One of the main features of such Lie algebras is that they are graded by finite root systems possessing nondegenerate symmetric invariant bilinear forms. The EALAs, including toroidal Lie algebras and the central extensions of the matrix Lie algebras coordinated by quantum tori as   well-known examples, have been subjects of intensive studies during the last decade (e.g., \cite{AG,AABGP,B,BGK,BGKE,BL,G1,G2,GZ,HT,LS,M} and references therein). More precisely, the structure  and representation theories of EALAs have been well developed and  received much attention via various module realizations (e.g., \cite{BGK,BGKE,G1,G2,GZ,JM,LS}).
Despite of the fact that much progress has been achieved on these aspects, it seems that, to the best of our knowledge,
not much has yet been known on the aspect of quantum groups associated with EALAs.

Quantizing Lie algebras is an important approach to produce
new quantum groups. To quantize a Lie algebra,
it is necessary to first know its bialgebra structures. Therefore, investigating Lie bialgebra structures is an important problem and it is our goal to study Lie bialgebra structures on EALAs.
In the present paper, we initiate the study by considering Lie bialgebra structures on the extended affine Lie algebra $\wsl$. We wish the study may provide us an approach to investigate  Lie bialgebra structures on more general EALAs.

In order to obtain solutions of quantum Yang-Baxter equations, which plays important roles in mathematical physics,  Drinfeld \cite{D1,D2} introduced the notion of Lie bialgebras.  Lie bialgebra structures on the Witt algebra and the Virasoro algebra were considered and classified in \cite{T,NT}. Since then, Lie bialgebra structures as well as their quantizations on some infinite-dimensional graded Lie algebras, in particular those containing the Virasoro algebra (or its $q$-analog) have been extensively studied (e.g., \cite{CS,LSX,HLS,Mi,SS,SS1,T,WSS,WSS1,YS}).
We have noticed that, for the above-mentioned infinite-dimensional graded Lie algebras,
 the Virasoro algebra plays a very crucial role in determining their bialgebra structures because of the fact that the modules of the intermediate series of the Virasoro algebra have relatively simple module structures (see, e.g., \cite{SZ}). However, for our case here, due to the fact that the Lie algebra $\wsl$ has a rather complicated interior algebraic structure, some techniques used in determining  bialgebra structures of the above-mentioned   Lie algebras cannot be applied to the case for $\wsl$.
More precisely, when we determine the derivations from  $\wsl$  to its tensor product (which is indispensable in determining  bialgebra structures), we encounter the following 2 difficulties:
Firstly, we have the difficulty to solve some complicated systems of equations with  lots of undetermined variables (see, e.g., \eqref{equa-g01-0}--\eqref{equa-gh-ef}). Secondly, we have the difficulty to choose some suitable inner derivation, denoted $u_{\rm
inn}$, apart from some obvious ones (such as those with $u$ being in \eqref{Inn-0}),  so that when we subtract a derivation $\dir_0$ by the chosen inner derivation $u_{\rm
inn}$, the resulting derivation $\dir_0-u_{\rm
inn}$ would be as simple as possible.
In dealing with problems of determining derivations in this paper, we overcome the first difficulty by employing some new techniques and arguments (cf.~proofs of Claims \ref{claim1} and \ref{claim2}).
Then we overcome the second difficulty by finding two rather complicated inner derivations $u_{\rm
inn},\,v_{\rm
inn}$ in two steps, so that $\dir_0-u_{\rm
inn}-v_{\rm
inn}$ has a relatively simple form (cf.~\eqref{inner-u} and  \eqref{inner-v}).
We would like to remark that our techniques and arguments may be possibly generalized to more general EALAs. This is also one of our motivations to present the results here.

The main result of this paper can be formulated as follows.
\begin{theo}\label{main}
Every Lie bialgebra structure on $\wsl$ is triangular coboundary.
\end{theo}
\vs{5pt}

\ll{\bf\large{2\ \ \
Preliminary}}\setcounter{section}{2}\setcounter{equation}{0}\setcounter{theo}{0}\vs{4pt}

\noindent The extended affine Lie algebra $\wsl$ 
is essentially a loop algebra with the algebra of Laurent polynomials
replaced by a quantum torus (cf. \cite{MY}). We now give the precise definition below. Let $q$ be a nonzero complex number which is not a root of unity, and
let $\C_q=\Cq$ be the $\C$-algebra  defined by the
generators $x^{\pm1}$, $y^{\pm1}$ and
relations $x^{\pm1}x^{\mp1}=1$, $y^{\pm1}y^{\mp1}=1$, $yx=qxy$.
First, we need to recall the definition of the Lie algebra $\LL=\sl$.
Denote by $E_{ij}$ the $2\times 2$ matrix with the $(i,j)$-entry being
equal to 1 and others being $0$ for $i,j=1,2$.
We use $\Z^*$, $\Z^+$ and $\Z^-$ to  denote respectively the sets of all nonzero,
 nonnegative and nonpositive integers. We also denote
 $\bz=(0,0)$, $\BZ=\Z\times\Z$ and $\BZ^*=\BZ\sm\{\bz\}$. Then the following elements with $\bm=(m_1,m_2)\in\BZ$, $\bk=(k_1,k_2)\in\BZ^*$,
\begin{eqnarray*}
&&e_{\bm}=
E_{12}x^{m_1}y^{m_2},\ \ \ f_{\bm}=
E_{21}x^{m_1}y^{m_2},\\
&&d=E_{11}-E_{22},\ \ \ \ \ \ \  g_{\bk}=E_{11}x^{k_1}y^{k_2},\ \ \ \ \ \ \ \
h_{\bk}=E_{22}x^{k_1}y^{k_2},
\end{eqnarray*}
form a basis of $\LL$  satisfying the following relations,
for $\bm'=(m'_1,m'_2)\in\BZ$, $\bk'=(k'_1,k'_2)\in\BZ^*$,
\begin{eqnarray}\label{re-ss}
&&[\em,e_{\bm'}]=[\fm,f_{\bm'}]=[d,h_{\bk}]=[d,g_{\bk}]=[g_{\bk},h_{\bk'}]=0,\nonumber\\
&&[g_{\bk},\em]=q^{k_2m_1}e_{\bk+\bm},\ \ [h_{\bk},\em]=-q^{k_1m_2}e_{\bk+\bm},\ \ [d,\em]=2\em,\nonumber\\
&&[h_{\bk},\fm]=q^{k_2m_1}f_{\bk+\bm},\ \ [g_{\bk},\fm]=-q^{k_1m_2}f_{\bk+\bm},\ \ [d,\fm]=-2\fm,\nonumber\\
&&[\em,f_{-\bm}]=q^{m_2m'_1}d,\mbox{ \ \ \ }
[\em,f_{\bm'}]=
q^{m_2m'_1}g_{\bm+\bm'}-q^{m'_2m_1}h_{\bm+\bm'}\mbox{ if} \ \bm+\bm'\neq\bz,
\nonumber\\
&&[g_{\bk},g_{\bk'}]=
(q^{k_2k'_1}-q^{k'_2k_1})g_{\bk+\bk'}, \ \ \ \ \ \
[h_{\bk},h_{\bk'}]=
(q^{k_2k'_1}-q^{k'_2k_1})h_{\bk+\bk'}.
\end{eqnarray}
Thus,  $\LL=\oplus_{(m_1,m_2)\in\BZ}\LL_{m_1,m_2}$ is $\BZ$-graded
  with
\begin{eqnarray*}
&&\LL_{0,0}=\C\EE_{0,0}\oplus\C\FF_{0,0}\oplus\C\DD,\\[3pt]
&&\LL_{k_1,k_2}=\C\EE_{k_1,k_2}\oplus\C\FF_{k_1,k_2}\oplus\C
\GG_{k_1,k_2}\oplus\C\HH_{k_1,k_2}\ \ {\rm if}\ \ (k_1,k_2)\in\BZ^*.
\end{eqnarray*}
We introduce two {\it degree derivations} $d_1$ and $d_2$ on $\LL$, written in terms of brackets as below,
$$[d_1,L_{m_1,m_2}]=m_1L_{m_1,m_2},\ \,[d_2,L_{m_1,m_2}]=m_2L_{m_1,m_2},\ \, [d_1,d_2]=[\DD,d_1]=[\DD,d_2]=0,$$
for $L_{m_1,m_2}\in\LL_{m_1,m_2}.$ Then the extended affine Lie algebra $\wsl$  is simply the Lie algebra  $\LL\oplus\C d_1\oplus\C d_2$, which is denoted by $\WL$ throughout the paper. Obviously, $\WL=\oplus_{(m_1,m_2)\in\BZ}\WL_{m_1,m_2}$ is also $\BZ$-graded with $\WL_{k_1,k_2}=\LL_{k_1,k_2}$ if $(k_1,k_2)\in\BZ^*$, and $\WL_{0,0}=\LL_{0,0}\oplus\C d_1\oplus\C d_2$.

Now let us recall the definitions related to Lie bialgebras.
For any $\C$-vector space $\SS$, denote by
$\xi$ the {\it cyclic map} of $\SS\otimes\SS\otimes \SS$ cyclically
permuting the coordinates, namely, $ \xi (x_{1} \otimes x_{2}
\otimes x_{3}) =x_{2} \otimes x_{3} \otimes x_{1}$ and by $\tau$ the {\it twist map} of $\SS\otimes\SS$, i.e., $\tau(x_1\otimes x_2)=x_2 \otimes x_1$ for any $x_1,x_2,x_3\in\SS$.
First we need to reformulate the definitions of Lie algebras
and Lie coalgebras as follows:
A {\it Lie algebra} is a pair $(\SS,\delta)$ of a vector space
$\SS$ and a linear map $\delta :\SS\otimes\SS\rar\SS$ satisfying
\begin{eqnarray*}
\!\!\!\!\!\!\!\!\!\!\!\!
&&\Ker\,(1-\tau) \subset \Ker\,\delta,\\
\!\!\!\!\!\!\!\!\!\!\!\!&& \delta \cdot (1 \otimes
\delta ) \cdot (1 + \xi +\xi^{2}) =0:\ \ \SS \otimes\SS\otimes\SS\rar
\SS.
\end{eqnarray*}
Dually, a {\it Lie coalgebra} is a pair $(\SS,\D)$ of a vector space $\SS$
and a linear map $\D:\SS\to\SS\otimes\SS$ satisfying
\begin{eqnarray}
\!\!\!\!\!\!\!\!\!\!\!\!&&
\Im\,\D \subset \Im(1- \tau),\label{cLie-s-s}\\
\!\!\!\!\!\!\!\!\!\!\!\!&& (1 + \xi +\xi^{2}) \cdot
(1 \otimes \D) \cdot \D =0:\ \ \SS\to\SS\otimes\SS\otimes\SS.\nonumber
\end{eqnarray}
For a Lie algebra $\SS$, we also use $[x,y]=\delta(x,y)$ to denote its
Lie bracket and the symbol ``$\cdot$'' to stand for the {\it
diagonal adjoint action}:
$ 
x\cdot( a\otimes b)=
[x,a]\otimes
b+a\otimes[x,b].
$ 
\begin{defi}\rm
A {\it Lie bialgebra} is a triple $(\SS,\delta,\D)$ satisfying
\begin{eqnarray}
\!\!\!\!\!\!\!\!\!\!\!\!
&&\mbox{$(\SS, \delta)$ is a
Lie algebra},\ \ \mbox{$(\SS,\D)$ is a Lie coalgebra},\nonumber\\
\!\!\!\!\!\!\!\!\!\!\!\!\
&&\D\delta(x,y)=x\cdot\D y-y\cdot\D x,\ \ \forall\,\,x,y\in\SS\ \mbox{$($\it compatibility
condition$)$}.\label{bLie-d}
\end{eqnarray}
\end{defi}

Denote by $\UU$ the universal enveloping algebra of $\SS$ and by $1$
the identity element of $\UU$. For any $r =\sum_{i} {a_{i} \otimes
b_{i}}\in\SS\otimes\SS$, define $r^{12},\,r^{13},\,r^{23}$ to be
elements of $\UU \otimes \UU \otimes \UU$ by
\begin{eqnarray*}
r^{12}=\mbox{$\sum \limits_{i}$}{a_{i} \otimes b_{i}
\otimes 1},\ \ r^{13}= \mbox{$\sum \limits_{i}$} {a_{i} \otimes 1
\otimes b_{i}},\ \ r^{23}=\mbox{$\sum \limits_{i}$}{1 \otimes a_{i} \otimes
b_{i}},
\end{eqnarray*}
and $\mbox{$c(r)=[r^{12},r^{13}]+[r^{12},r^{23}]+[r^{13},r^{23}]$}$.

\begin{defi}\label{def2}\rm \begin{itemize}
\item[{\rm(1)}] A {\it coboundary Lie bialgebra} is a $4$-tuple
$(\SS, \delta, \D,r),$ where $(\SS,\delta,\D)$ is a Lie bialgebra
and $r \in \Im(1-\tau) \subset\SS\otimes\SS$ such that $\D=\D_r$ is
a {\it coboundary of $r$}, i.e.,
\begin{eqnarray}\label{coboun}
\label{D-r}\D_r(x)=x\cdot r\mbox{\ \ for\ \ }x\in\SS.
\end{eqnarray}
\item[{\rm (2)}] A coboundary Lie bialgebra $(\SS,\delta,\D,r)$
is called {\it triangular} if it satisfies the following {\it
classical Yang-Baxter Equation} ({\it CYBE\,})
\begin{eqnarray}
c(r)=0.\label{CYBE}
\end{eqnarray}
\end{itemize}
\end{defi}
\vs{5pt}

\ll{\bf\large{3\ \ \ Proof of the main results}}\setcounter{section}{3}
\setcounter{theo}{0}\setcounter{equation}{0}
\vs{6pt}

\noindent The aim of this section is to give a proof of Theorem \ref{main}.
First one has the following  \cite{D1,D2,NT}.\vs{-6pt}
\begin{lemm}\label{some}
Let $\SS$ be a Lie algebra and $r\in\Im(1-\tau)\subset
\SS\otimes\SS$.

{\rm(1)} The triple $(\SS,[\cdot,\cdot], \D_r)$ is
a Lie bialgebra if and only if $r$ satisfies CYBE $(\ref{CYBE})$.

{\rm(2)} For any $x\in\SS$, $(1+\xi+\xi^{2})\cdot(1\otimes\D)\cdot\D(x)=x\cdot
c(r)$.
\end{lemm}
Similar to the arguments in \cite[Lemma 2.2]{WSS}, one can obtain the following.
\begin{lemm}\label{Legr}
Regard $\WL^{\otimes n}$, the tensor product of $n$ copies of $\WL$, as an $\WL$-module under the adjoint diagonal action of $\WL$. If $x\cdot r=0$ for some $r\in\WL^{\otimes n}$ and all $x\in\WL$, then $r=0$.
\end{lemm}

An element $r\in\Im(1-\tau)\subset\WL\otimes\WL$ is said to satisfy
the \textit{modified Yang-Baxter equation} ({\it MYBE\,}) if
\begin{eqnarray}
x\cdot c(r)=0,\ \,\ \forall\,\,x\in\WL.\label{MYBE}
\end{eqnarray}
As a conclusion of Lemma \ref{Legr}, we immediately obtain
\begin{coro}\label{coro1}
Any element $r\in\Im(1-\tau)\subset\WL\otimes\WL$ satisfies $(\ref{CYBE})$ if and only if it satisfies MYBE $(\ref{MYBE})$.
\end{coro}

Regard $\VV=\WL\otimes\WL$ as an $\WL$-module under the adjoint
diagonal action. Denote by $\Der(\WL,\VV)$ the set of
\textit{derivations} $\dir:\WL\to\VV$, i.e., $\dir$ is a linear map
satisfying
\begin{eqnarray}
\dir([x,y])=x\cdot \dir(y)-y\cdot \dir(x)\mbox{ \ for \ }x,y\in\WL,\label{deriv}
\end{eqnarray}
and $\Inn(\WL,\VV)$ the set consisting of   {\it inner derivations} $v_{\rm
inn}, \, v\in\VV$, defined by
\begin{eqnarray}
\label{inn} v_{\rm inn}:x\mapsto x\cdot v\mbox{ \ for \ }x\in\WL.
\end{eqnarray}
Denote by $H^1(\WL,\VV)$  the {\it first cohomology group} of $\WL$ with coefficients in the $\WL$-module $\VV$. Then
\begin{eqnarray*}
H^1(\WL,\VV)\cong\Der(\WL,\VV)/\Inn(\WL,\VV).
\end{eqnarray*}

\begin{prop}\label{proposition}We have
$\Der(\WL,\VV)=\Inn(\WL,\VV)$, i.e., $H^1(\WL,\VV)=0$.
\end{prop}

We shall prove this proposition by several lemmas.
First note that $\VV=\WL\otimes\WL=\oplus_{\bm\in\BZ}\VV_\bm$ is $\BZ$-graded with
\begin{eqnarray*}
\VV_\bm=\mbox{$\sum\limits_{\bn+\bn'=\bm}$}\WL_\bn\otimes\WL_{\bn'}\ \ \,\,\mbox{for}\ \,\,\bn,\,\bn'\in\BZ.
\end{eqnarray*}
A derivation $\dir\in\Der(\WL,\VV)$ is {\it homogeneous of degree
$\bm$} if $\dir(\WL_\bn)\subset \WL_{\bm+\bn}$ for all $\bn\in\BZ$.
Denote
\begin{eqnarray*}
\Der(\WL,\VV)_\bm=\{\dir\in\Der(\WL,\VV)\,|\,{\rm deg\,}\dir=\bm\}\mbox{ \ \ for \ } \bm\in\BZ.
\end{eqnarray*}
For any fixed $\dir\in\Der(\WL,\VV)$, and any $\bm\in\BZ$, we define a linear map
$\dir_\bm:\WL\rightarrow\VV$ as follows: Let $\mu\in\WL_\bn$,
$\bn\in\BZ$, we can write $\dir(\mu)=\sum_{\bm'\in\BZ}\mu_{\bm'}$ with $\mu_{\bm'}\in\VV_{\bm'}$,
then we set $\dir_\bm(\mu)=\mu_{\bn+\bm}$. Obviously, $\dir_\bm\in
\Der(\WL,\VV)_\bm$ and we have
\begin{eqnarray}\label{summable}
\dir=\mbox{$\sum\limits_{\bm\in\BZ}\dir_\bm$},
\end{eqnarray}
which holds in the sense that for every $\mu\in\WL$, only finitely
many $\dir_\bm(\mu)\neq 0,$ and $\dir(u)=\sum_{\bm\in\BZ}\dir_\bm(\mu)$
(we call such a
sum in (\ref{summable}) {\it summable}).\\[8pt]

\begin{lemm}\label{lemm-Dn}
If $\bk\in\BZ^*$, then $\dir_\bk\in\Inn(\WL,\VV)$.
\end{lemm}
\begin{proof}
Denote $\texttt{T}={\rm Span}_{\C}\{d_1,d_2\}.$ Define the nondegenerate bilinear map $(\cdot,\cdot):\BZ\times\texttt{T} \to \C$ by $(\bm,\rho)=m_{1}\rho_{1}+m_{2}\rho_{2}$ for $\bm=(m_{1},m_{2})\in\BZ,$ $\rho=\rho_{1}d_1+\rho_{2}d_2\in\texttt{T}.$ We simply denote $\rho(\bm)=(\bm,\rho)$. Then for $\bk\in\BZ^*$, by linear algebra, one can choose $\rho\in\texttt{T}$ with $\rho(\bk)\neq0.$ Now for any $x\in\WL_{\bm}$, $\bm\in\BZ$, by applying $\dir_\bn$ to $[\rho,x]=\rho(\bm)x$, and using $\dir_\bk(x)\in\VV_{\bk+\bm}$, we have $$\rho(\bk+\bm)\dir_\bk(x)-x\cdot\dir_\bk(\rho)=\rho\cdot\dir_\bk(x)-x\cdot\dir_\bk(\rho)=\rho(\bm)\dir_\bk(x),$$ i.e., $\dir_\bk(x)=v_{\rm inn}(x)$ with
$v=\frac1{\rho(\bk)}\dir_\bk(\rho)\in\WL_{\bk}.$  Thus $\dir_\bk\in\Inn(\WL,\VV)$.
\end{proof}

\begin{lemm}\label{lemm-D12} We have
$\dir_\bz(d_1)=\dir_\bz(d_2)=0$.
\end{lemm}
\begin{proof}~Applying $\dir_\bz$ to $[d_1,x]=m_{1}x$ and $[d_2,x]=m_{2}x$ for $x\in\WL_{m_{1},m_{2}},$ we obtain that $x\cdot\dir_\bz(d_{1})=x\cdot\dir_\bz(d_{2})=0$. Thus by Lemma \ref{Legr}, $\dir_\bz(d_1)=\dir_\bz(d_2)=0.$
\end{proof}

\begin{lemm}
\label{lemm-L0}  By replacing $\dir_\bz$ by $\dir_\bz-u_{\rm inn}$ for some $u\in
\VV_\bz$, one can suppose $\dir_\bz(\WL_{\bz})=0$.
\end{lemm}
\begin{proof} Write $\dir_\bz(\DD)$ as (here and below, we always assume that the sums are over
a finite number of $\bm\in\BZ$ and $\bk\in\BZ^*$)
{\small\begin{eqnarray}\label{first-}\nonumber
\!\!\!\!\!\!\!\!\!\!\!\!\!\!\!\!\!&\!\!\!\!\!\!\dir_\bz(\DD)=\!\!\!&\mbox{$\sum\limits_{\bm}$}\big(\d^{ee}_{\bm}\EE_{\bm}\!\otimes\!\EE_{-\bm}
\!+\d^{ef}_{\bm}\EE_{\bm}\!\otimes\!\FF_{-\bm}\!+\d^{fe}_{\bm}\FF_{\bm}\!\otimes\!\EE_{-\bm}\!
+\d^{ff}_{\bm}\FF_{\bm}\!\otimes\!\FF_{-\bm}\big)
\!+\!\mbox{$\sum\limits_{\bk}$}\big(\d^{eg}_{\bk}\EE_{\bk}\!\otimes\!\GG_{-\bk}
\\[-5pt]
\nonumber\!\!\!\!\!\!\!\!\!\!\!\!\!\!\!\!\!\!\!\!\!\!\!\!\!\!\!\!\!\!&&+\d^{eh}_{\bk}\EE_{\bk}\!\otimes\!\HH_{-\bk}\!\!+\d^{he}_{\bk}\HH_{\bk}\!\otimes\!\EE_{-\bk}\!\!
+\d^{ge}_{\bk}\GG_{\bk}\!\otimes\!\EE_{-\bk}\!\!+\d^{fg}_{\bk}\FF_{\bk}\!\otimes\!\GG_{-\bk}
\!+\d^{fh}_{\bk}\FF_{\bk}\!\otimes\!\HH_{-\bk}\!+\d^{gf}_{\bk}\GG_{\bk}\!\otimes\!\FF_{-\bk}\\
\nonumber\!\!\!\!\!\!\!\!\!\!\!\!\!\!\!\!\!\!\!\!\!\!\!\!&&+\d^{hf}_{\bk}\HH_{\bk}\!\otimes\!\FF_{-\bk}\!+\d^{gg}_{\bk}\GG_{\bk}\!\otimes\!\GG_{-\bk}
\!+\d^{gh}_{\bk}\GG_{\bk}\!\otimes\!\HH_{-\bk}+\d^{hg}_{\bk}\HH_{\bk}\!\otimes\!\GG_{-\bk}\!
+\d^{hh}_{\bk}\HH_{\bk}\otimes\HH_{-\bk}\big)\\
\nonumber\!\!\!\!\!\!\!\!\!\!\!\!\!\!\!\!\!\!\!\!\!\!\!\!\!\!\!\!\!\!&&+\d^e_d\EE_{\bz}\otimes\DD\!+\d^e_{d_{1}}\EE_{\bz}\otimes\DD_{1}
\!+\d^e_{d_{2}}\EE_{\bz}\otimes\DD_{2}\!+\d^f_d\FF_{\bz}\otimes\DD
\!+\d^f_{d_{1}}\FF_{\bz}\otimes\DD_{1}\!+\d^f_{d_{2}}\FF_{\bz}\!\otimes\!\DD_{2}
\!+\d^d_e\DD\!\otimes\!\EE_{\bz}\\
\nonumber\!\!\!\!\!\!\!\!\!\!\!\!\!\!\!\!\!\!\!\!\!\!\!\!\!\!\!\!\!\!&&+\d^d_f\DD\otimes\FF_{\bz}+\d^d_d\DD\!\otimes\!\DD+\d^d_{d_{1}}\DD\!\otimes\!\DD_{1}
+\d^d_{d_{2}}\DD\otimes\DD_{2}+\d^{d_{1}}_e\DD_{1}\!\otimes\!\EE_{\bz}\!+\d^{d_{1}}_f\DD_{1}\!\otimes\!\FF_{\bz}\!
+\d^{d_{1}}_d\DD_{1}\!\otimes\!\DD\\
\!\!\!\!\!\!\!\!\!\!\!\!\!\!\!\!\!\!\!\!\!\!\!\!\!\!\!\!\!\!&&+\d^{d_{1}}_{d_{1}}\DD_{1}\!\otimes\!\DD_{1}\!\!+\d^{d_{1}}_{d_{2}}\DD_{1}\!\otimes\!\DD_{2}\!
+\d^{d_{2}}_e\DD_{2}\!\otimes\!\EE_{\bz}\!+\d^{d_{2}}_f\DD_{2}\!\otimes\!\FF_{\bz}\!
+\d^{d_{2}}_d\DD_{2}\!\otimes\!\DD\!+\d^{d_{2}}_{d_{1}}\DD_{2}\!\otimes\!\DD_{1}
\!+\d^{d_{2}}_{d_{2}}\DD_{2}\!\otimes\!\DD_{2},
\end{eqnarray}}%
where the coefficients associated with the Greek symbol $\d$ are all in $\C[q^{\pm1}]$ (here and below,
we always use the Greek symbols $\lambda,\mu,\nu,\eta,\xi,\zeta,\gamma$
 with some meaningful superscripts and subscripts to denote undetermined
 coefficients in $\C[q^{\pm1}]$). We need to use the following identities
{\small\begin{eqnarray*}\begin{array}{llll}
\DD\cdot(\EE_{\bk}\otimes\GG_{-\bk})=2\EE_{\bk}\otimes\GG_{-\bk},&\ \ \ \ & \DD\cdot(\FF_{\bk}\otimes\GG_{-\bk})=-2\FF_{\bk}\otimes\GG_{-\bk},\\[8pt]
\DD\cdot(\GG_{\bk}\otimes\EE_{-\bk})=2\GG_{\bk}\otimes\EE_{-\bk},&\ \ \ \ &
\DD\cdot(\GG_{\bk}\otimes\FF_{-\bk})=-2\GG_{\bk}\otimes\FF_{-\bk},\\[8pt]
\DD\cdot(\EE_{\bk}\otimes\HH_{-\bk})=2\EE_{\bk}\otimes\HH_{-\bk},&\ \ \ \ &
\DD\cdot(\FF_{\bk}\otimes\HH_{-\bk})=-2\FF_{\bk}\otimes\HH_{-\bk},\\[8pt]
\DD\cdot(\HH_{\bk}\otimes\EE_{-\bk})=2\HH_{\bk}\otimes\EE_{-\bk},&\ \ \ \ &
\DD\cdot(\HH_{\bk}\otimes\FF_{-\bk})=-2\HH_{\bk}\otimes\FF_{-\bk},\\[8pt]
\DD\cdot(\EE_{\bm}\otimes\EE_{-\bm})=4\EE_{\bm}\otimes\EE_{-\bm},&\ \ \ \ &
\DD\cdot(\FF_{\bm}\otimes\FF_{-\bm})=-4\FF_{\bm}\otimes\FF_{-\bm},\end{array}\\
\begin{array}{lll}
\DD\cdot(\EE_{\bz}\otimes\DD)=2\EE_{\bz}\otimes\DD,&
 \DD\cdot(\EE_{\bz}\otimes d_{1})=2\EE_{\bz}\otimes d_{1},&
\DD\cdot(\EE_{\bz}\otimes d_{2})=2\EE_{\bz}\otimes d_{2},\\[8pt]
\DD\cdot(\DD\otimes\EE_{\bz})=2\DD\otimes\EE_{\bz},
& \DD\cdot( d_{1}\otimes\EE_{\bz})=2d_{1}\otimes\EE_{\bz},& \DD\cdot( d_{2}\otimes\EE_{\bz})=2d_{2}\otimes\EE_{\bz},\\[8pt]
\DD\cdot(\FF_{\bz}\otimes\DD)=-2\FF_{\bz}\otimes\DD,& \DD\cdot(\FF_{\bz}\otimes d_{1})=-2\FF_{\bz}\otimes d_{1},& \DD\cdot(\FF_{\bz}\otimes d_{2})=-2\FF_{\bz}\otimes d_{2},\\[8pt]
\DD\cdot(\DD\otimes\!\FF_{\bz})=-2\DD\otimes\FF_{\bz},& \DD\cdot( d_{1}\otimes\FF_{\bz})=-2d_{1}\otimes\FF_{\bz},& \DD\cdot( d_{2}\otimes\FF_{\bz})=-2d_{2}\otimes\FF_{\bz}.\end{array}
\end{eqnarray*}\small}%
In particuler, by replacing $\dir_\bz$ by $\dir_\bz-u_{\rm inn},$ where $u$ is some combination
of
\begin{eqnarray}\label{Inn-0}
&&\!\!\!\!\!\!\!\!\!\!\!\!\!\!\!\!
\EE_{\bm}\otimes\EE_{-\bm},\ \ \EE_{\bk}\otimes\GG_{-\bk},\ \ \EE_{\bk}\otimes\HH_{-\bk},\ \ \FF_{\bm}\otimes\FF_{-\bm},\ \ \FF_{\bk}\otimes\GG_{-\bk},\ \ \FF_{\bk}\otimes\HH_{-\bk},\ \ \GG_{\bk}\otimes\EE_{-\bk},\nonumber\\
&&\!\!\!\!\!\!\!\!\!\!\!\!\!\!\!\!
\HH_{\bk}\otimes\EE_{-\bk},\ \ \GG_{\bk}\otimes\FF_{-\bk},\ \ \HH_{\bk}\otimes\FF_{-\bk},\ \ \EE_{\bz}\otimes\DD,\ \ \EE_{\bz}\otimes\DD_{1},\ \ \EE_{\bz}\otimes\DD_{2},\ \ \FF_{\bz}\otimes\DD,\ \ \FF_{\bz}\otimes\DD_{1},\nonumber\\
&&\!\!\!\!\!\!\!\!\!\!\!\!\!\!\!\!
\FF_{\bz}\otimes\DD_{2},\ \ \DD\otimes\EE_{\bz},\ \ \DD\otimes\FF_{\bz},\ \ \DD_{1}\otimes\EE_{\bz},\ \
\DD_{1}\otimes\FF_{\bz},\ \ \DD_{2}\otimes\EE_{\bz},\ \ \DD_{2}\otimes\FF_{\bz},\end{eqnarray}
 we can  suppose
\begin{eqnarray*}
\d^{ee}_{\bm}\!\!\!&=&\!\!\d^{eg}_{\bk}=\d^{eh}_{\bk}=\d^e_d=\d^e_{d_{1}}=\d^e_{d_{2}}=\d^{ff}_{\bm}
=\d^{fg}_{\bk}=\d^{fh}_{\bk}=\d^f_d=\d^f_{d_{1}}\\
\!\!\!&=&\!\!\d^f_{d_{2}}=\d^{ge}_{\bk}=\d^{gf}_{\bk}=\d^{he}_{\bk}=\d^{hf}_{\bk}=\d^d_e\!=\d^d_f\!=\d^{d_{1}}_e\!=\d^{d_{1}}_f\!=\d^{d_{2}}_e\!=\d^{d_{2}}_f=0.
\end{eqnarray*}
Thus $\dir_\bz(\DD)$ can be simplified as
\begin{eqnarray*}
&\!\!\!\!\!\!\dir_\bz(\DD)=\!\!\!&\mbox{$\sum\limits_{\bm}$}\big(\d^{ef}_{\bm}\EE_{\bm}\otimes\FF_{-\bm}
+\d^{fe}_{\bm}\FF_{\bm}\otimes\EE_{-\bm}\big)+\d^d_d\DD\otimes\DD
+\d^d_{d_{1}}\DD\otimes\DD_{1}+\d^d_{d_{2}}\DD\otimes\DD_{2}\\[-5pt]
&&+\mbox{$\sum\limits_{\bk}$}\big(\d^{gg}_{\bk}\GG_{\bk}\otimes\GG_{-\bk}
+\d^{gh}_{\bk}\GG_{\bk}\otimes\HH_{-\bk}+\d^{hg}_{\bk}\HH_{\bk}\otimes\GG_{-\bk}
+\d^{hh}_{\bk}\HH_{\bk}\otimes\HH_{-\bk}\big)\\[-5pt]
&&+\d^{d_1}_d\DD_{1}\otimes\DD+\d^{d_{1}}_{d_{1}}\DD_{1}\otimes\DD_{1}
+\d^{d_{1}}_{d_{2}}\DD_{1}\otimes\DD_{2}+\d^{d_{2}}_d\DD_{2}\otimes\DD
+\d^{d_{2}}_{d_{1}}\DD_{2}\otimes\DD_{1}
+\d^{d_{2}}_{d_{2}}\DD_{2}\otimes\DD_{2}.
\end{eqnarray*}
Similarly, write (cf.~statement after \eqref{first-})
{\small\begin{eqnarray*}
&\!\!\!\!\!\!\dir_\bz(\EE_{\bz})=\!\!\!&\mbox{$\sum\limits_{\bm}$}\big(\e^{ee}_{\bm}\EE_{\bm}\!\otimes\!\EE_{-\bm}
\!+\e^{ef}_{\bm}\EE_{\bm}\!\otimes\!\FF_{-\bm}\!+\e^{fe}_{\bm}\FF_{\bm}\!\otimes\!\EE_{-\bm}\!
+\e^{ff}_{\bm}\FF_{\bm}\!\otimes\!\FF_{-\bm}\big)
\!+\!\mbox{$\sum\limits_{\bk}$}\big(\e^{eg}_{\bk}\EE_{\bk}\!\otimes\!\GG_{-\bk}\\[-5pt]
&&+\e^{eh}_{\bk}\EE_{\bk}\!\otimes\!\HH_{-\bk}\!\!+\e^{he}_{\bk}\HH_{\bk}\!\otimes\!\EE_{-\bk}\!\!
+\e^{ge}_{\bk}\GG_{\bk}\!\otimes\!\EE_{-\bk}\!\!+\e^{fg}_{\bk}\FF_{\bk}\!\otimes\!\GG_{-\bk}
\!+\e^{fh}_{\bk}\FF_{\bk}\!\otimes\!\HH_{-\bk}\!+\e^{gf}_{\bk}\GG_{\bk}\!\otimes\!\FF_{-\bk}\\
&&+\e^{hf}_{\bk}\HH_{\bk}\!\otimes\!\FF_{-\bk}\!+\e^{gg}_{\bk}\GG_{\bk}\!\otimes\!\GG_{-\bk}
\!+\e^{gh}_{\bk}\GG_{\bk}\!\otimes\!\HH_{-\bk}+\e^{hg}_{\bk}\HH_{\bk}\!\otimes\!\GG_{-\bk}\!
+\e^{hh}_{\bk}\HH_{\bk}\otimes\HH_{-\bk}\big)\\
&&+\e^e_d\EE_{\bz}\otimes\DD\!+\e^e_{d_{1}}\EE_{\bz}\otimes\DD_{1}
\!+\e^e_{d_{2}}\EE_{\bz}\otimes\DD_{2}\!+\e^f_d\FF_{\bz}\otimes\DD
\!+\e^f_{d_{1}}\FF_{\bz}\otimes\DD_{1}\!+\e^f_{d_{2}}\FF_{\bz}\!\otimes\!\DD_{2}
\!+\e^d_e\DD\!\otimes\!\EE_{\bz}\\
&&+\e^d_f\DD\otimes\FF_{\bz}+\e^d_d\DD\!\otimes\!\DD+\e^d_{d_{1}}\DD\!\otimes\!\DD_{1}
+\e^d_{d_{2}}\DD\otimes\DD_{2}+\e^{d_{1}}_e\DD_{1}\!\otimes\!\EE_{\bz}\!+\e^{d_{1}}_f\DD_{1}\!\otimes\!\FF_{\bz}\!
+\e^{d_{1}}_d\DD_{1}\!\otimes\!\DD\\
&&+\e^{d_{1}}_{d_{1}}\DD_{1}\!\otimes\!\DD_{1}\!\!+\!\e^{d_{1}}_{d_{2}}\DD_{1}\!\otimes\!\DD_{2}\!
+\e^{d_{2}}_e\DD_{2}\!\otimes\!\EE_{\bz}\!+\e^{d_{2}}_f\DD_{2}\!\otimes\!\FF_{\bz}\!
+\e^{d_{2}}_d\DD_{2}\!\otimes\!\DD\!+\e^{d_{2}}_{d_{1}}\DD_{2}\!\otimes\!\DD_{1}
\!+\e^{d_{2}}_{d_{2}}\DD_{2}\!\otimes\!\DD_{2}.
\end{eqnarray*}\small}%
Note that we have
\begin{eqnarray*}
&&\EE_{\bz}\cdot(\HH_{\bk}\otimes\GG_{-\bk})=\EE_{\bk}\otimes\GG_{-\bk}
-\HH_{\bk}\otimes\EE_{-\bk}, \ \ \ \ \EE_{\bz}\cdot(\FF_{\bk}\otimes\EE_{-\bk})=\GG_{\bk}\otimes\EE_{-\bk}-\HH_{\bk}\otimes\EE_{-\bk},\\[8pt]
&&\EE_{\bz}\cdot(\GG_{\bk}\otimes\GG_{-\bk})=-\EE_{\bk}\otimes\GG_{-\bk}
-\GG_{\bk}\otimes\EE_{-\bk},\ \ \EE_{\bz}\cdot(\HH_{\bk}\otimes\HH_{-\bk})=\EE_{\bk}\otimes\HH_{-\bk}+\HH_{\bk}\otimes\EE_{-\bk},\\[8pt]
&&\EE_{\bz}\cdot(\FF_{\bz}\otimes\EE_{\bz})=\DD\otimes\EE_{\bz},
\ \ \ \ \ \EE_{\bz}\cdot(\DD\otimes\!\DD_{1})=-2\EE_{\bz}\otimes\DD_{1}, \ \ \ \ \ \EE_{\bz}\cdot(\DD\otimes\DD_{2})=-2\EE_{\bz}\otimes\DD_{2},\\[8pt]
&&\EE_{\bz}\cdot(\EE_{\bz}\otimes\FF_{\bz})=\EE_{\bz}\otimes\DD,\ \ \ \ \ \EE_{\bz}\cdot(\DD_{1}\otimes\DD)=-2\DD_{1}\otimes\EE_{\bz},\ \ \ \ \ \EE_{\bz}\cdot(\DD_{2}\otimes\DD)=-2\DD_{2}\otimes\EE_{\bz}.
\end{eqnarray*}
The first four equations imply
\begin{eqnarray*}
\EE_{\bk}\otimes\HH_{-\bk}\!\!\!&=&\!\!\!\frac{1}{2}\EE_{\bz}\cdot(\GG_{\bk}\otimes\GG_{-\bk}
+\HH_{\bk}\otimes\GG_{-\bk}+2\HH_{\bk}\otimes\HH_{-\bk}+\FF_{\bk}\otimes\EE_{-\bk}),\\
\HH_{\bk}\otimes\EE_{-\bk}\!\!\!&=&\!\!\!-\frac{1}{2}\EE_{\bz}\cdot(\GG_{\bk}\otimes\GG_{-\bk}
+\FF_{\bk}\otimes\EE_{-\bk}+\HH_{\bk}\otimes\GG_{-\bk}),\\
\EE_{\bk}\otimes\GG_{-\bk}\!\!\!&=&\!\!\!-\frac{1}{2}\EE_{\bz}\cdot(\GG_{\bk}\otimes\GG_{-\bk}
+\FF_{\bk}\otimes\EE_{-\bk}-\HH_{\bk}\otimes\GG_{-\bk}),\\
\GG_{\bk}\otimes\EE_{-\bk}\!\!\!&=&\!\!\!-\frac{1}{2}\EE_{\bz}\cdot(\GG_{\bk}\otimes\GG_{-\bk}
+\HH_{\bk}\otimes\GG_{-\bk}-\FF_{\bk}\otimes\EE_{-\bk}).
\end{eqnarray*}
Thus, by replacing $\dir_\bz$ by $\dir_\bz-u_{\rm inn},$ where $u$ is some combination of $\GG_{\bk}\otimes\GG_{-\bk}$, $\HH_{\bk}\otimes\HH_{-\bk}$, $\HH_{\bk}\otimes\GG_{-\bk}$, $\FF_{\bm}\otimes\EE_{-\bm}$, $\EE_{\bz}\otimes\FF_{\bz}$, $\DD_1\otimes\DD$, $\DD\otimes\DD_{1}$, $\DD\otimes\DD_{2}$, and $\DD_{2}\otimes\DD,$ we can  suppose
{\small\begin{eqnarray*}
\dir_\bz(\EE_{\bz})\!\!\!&=\!\!\!&\mbox{$\sum\limits_{\bm}$}\big(\e^{ee}_{\bm}\EE_{\bm}\!\otimes\!\EE_{-\bm}
\!+\e^{ef}_{\bm}\EE_{\bm}\!\otimes\!\FF_{-\bm}\!+\e^{fe}_{\bm}\FF_{\bm}\!\otimes\!\EE_{-\bm}\!
+\e^{ff}_{\bm}\FF_{\bm}\!\otimes\!\FF_{-\bm}\big)
\!+\!\mbox{$\sum\limits_{\bk}$}\big(\e^{fg}_{\bk}\FF_{\bk}\!\otimes\!\GG_{-\bk}\\[-5pt]
&&+\e^{fh}_{\bk}\FF_{\bk}\!\otimes\!\HH_{-\bk}\!\!+\e^{gf}_{\bk}\GG_{\bk}\!\otimes\!\FF_{-\bk}
\!\!+\e^{hf}_{\bk}\HH_{\bk}\!\otimes\!\FF_{-\bk}\!\!+\e^{gg}_{\bk}\GG_{\bk}\!\otimes\!\GG_{-\bk}
\!\!+\e^{gh}_{\bk}\GG_{\bk}\!\otimes\!\HH_{-\bk}\!\!+\e^{hg}_{\bk}\HH_{\bk}\!\otimes\!\GG_{-\bk}\\
&&+\e^{hh}_{\bk}\HH_{\bk}\otimes\HH_{-\bk}\big)+\e^f_d\FF_{\bz}\otimes\DD
+\e^f_{d_{1}}\FF_{\bz}\otimes\DD_{1}+\e^f_{d_{2}}\FF_{\bz}\otimes\DD_{2}
+\e^d_f\DD\otimes\FF_{\bz}\!+\e^d_d\DD\otimes\DD\\
&&+\e^d_{d_{1}}\DD\otimes\DD_{1}
+\e^d_{d_{2}}\DD\otimes\DD_{2}+\e^{d_{1}}_f\DD_{1}\otimes\FF_{\bz}
+\e^{d_{1}}_d\DD_{1}\otimes\DD+\e^{d_{1}}_{d_{1}}\DD_{1}\otimes\DD_{1}
+\e^{d_{1}}_{d_{2}}\DD_{1}\otimes\DD_{2}\\
&&+\e^{d_{2}}_f\DD_{2}\otimes\FF_{\bz}
+\e^{d_{2}}_d\DD_{2}\otimes\DD+\e^{d_{2}}_{d_{1}}\DD_{2}\otimes\DD_{1}
+\e^{d_{2}}_{d_{2}}\DD_{2}\otimes\DD_{2}.
\end{eqnarray*}\small}%
Applying $\dir_\bz$ to $[\DD,\EE_{\bz}]=2\EE_{\bz},$ we obtain
\begin{eqnarray*}
\DD\cdot\dir_\bz(\EE_{\bz})-\EE_{\bz}\cdot\dir_\bz(\DD)=2\dir_\bz(\EE_{\bz}).
\end{eqnarray*}
Comparing both sides of the above equation, one sees that
\begin{eqnarray}
&&\dir_\bz(\EE_{\bz})=0,\label{equa-e0}\\
&&\d^d_{d_1}=\d^d_{d_2}=\d^{d_1}_d=\d^{d_2}_d=0,\nonumber\\
&&\d^{ef}_{\bk}-\d^{gg}_{\bk}+\d^{hg}_{\bk}=0,\ \ \d^{ef}_{\bk}+\d^{gh}_{\bk}-\d^{hh}_{\bk}=0,\ \ \d^{ef}_{\bz}-2\d^d_d=0,\nonumber\\
&&\d^{fe}_{\bk}+\d^{hg}_{\bk}-\d^{hh}_{\bk}=0,\ \ \d^{fe}_{\bk}-\d^{gg}_{\bk}+\d^{gh}_{\bk}=0,\ \ \d^{fe}_{\bz}-2\d^d_d=0.\label{equa-d0-1}
\end{eqnarray}
Write  (cf.~statement after \eqref{first-})
{\small\begin{eqnarray*}\dir_\bz(\FF_{\bz})
\!\!\!&=\!\!\!&
\mbox{$\sum\limits_{\bm}$}\big(\f^{ee}_{\bm}\EE_{\bm}\!\otimes\!\EE_{-\bm}
\!+\f^{ef}_{\bm}\EE_{\bm}\!\otimes\!\FF_{-\bm}\!+\f^{fe}_{\bm}\FF_{\bm}\!\otimes\!\EE_{-\bm}\!
+\f^{ff}_{\bm}\FF_{\bm}\!\otimes\!\FF_{-\bm}\big)
\!+\!\mbox{$\sum\limits_{\bk}$}\big(\f^{eg}_{\bk}\EE_{\bk}\!\otimes\!\GG_{-\bk}\nonumber\\[-5pt]
&&+\f^{eh}_{\bk}\EE_{\bk}\!\otimes\!\HH_{-\bk}\!\!+\f^{he}_{\bk}\HH_{\bk}\!\otimes\!\EE_{-\bk}\!\!
+\f^{ge}_{\bk}\GG_{\bk}\!\otimes\!\EE_{-\bk}\!\!+\f^{fg}_{\bk}\FF_{\bk}\!\otimes\!\GG_{-\bk}
\!+\f^{fh}_{\bk}\FF_{\bk}\!\otimes\!\HH_{-\bk}\!+\f^{gf}_{\bk}\GG_{\bk}\!\otimes\!\FF_{-\bk}\nonumber\\
&&+\f^{hf}_{\bk}\HH_{\bk}\!\otimes\!\FF_{-\bk}\!+\f^{gg}_{\bk}\GG_{\bk}\!\otimes\!\GG_{-\bk}
\!+\f^{gh}_{\bk}\GG_{\bk}\!\otimes\!\HH_{-\bk}+\f^{hg}_{\bk}\HH_{\bk}\!\otimes\!\GG_{-\bk}\!
+\f^{hh}_{\bk}\HH_{\bk}\otimes\HH_{-\bk}\big)\nonumber\\
&&+\f^e_d\EE_{\bz}\otimes\DD\!+\f^e_{d_{1}}\EE_{\bz}\otimes\DD_{1}
\!+\f^e_{d_{2}}\EE_{\bz}\otimes\DD_{2}\!+\f^f_d\FF_{\bz}\otimes\DD
\!+\f^f_{d_{1}}\FF_{\bz}\otimes\DD_{1}\!+\f^f_{d_{2}}\FF_{\bz}\!\otimes\!\DD_{2}
\!+\f^d_e\DD\!\otimes\!\EE_{\bz}\nonumber\\
&&+\f^d_f\DD\otimes\FF_{\bz}+\f^d_d\DD\!\otimes\!\DD+\f^d_{d_{1}}\DD\!\otimes\!\DD_{1}
+\f^d_{d_{2}}\DD\otimes\DD_{2}+\f^{d_{1}}_e\DD_{1}\!\otimes\!\EE_{\bz}\!+\f^{d_{1}}_f\DD_{1}\!\otimes\!\FF_{\bz}\!
+\f^{d_{1}}_d\DD_{1}\!\otimes\!\DD\nonumber\\
&&\!+\f^{d_{1}}_{d_{1}}\DD_{1}\!\otimes\!\DD_{1}\!\!+\!\f^{d_{1}}_{d_{2}}\DD_{1}\!\otimes\!\DD_{2}\!
+\f^{d_{2}}_e\DD_{2}\!\otimes\!\EE_{\bz}\!+\f^{d_{2}}_f\DD_{2}\!\otimes\!\FF_{\bz}\!
+\f^{d_{2}}_d\DD_{2}\!\otimes\!\DD\!+\f^{d_{2}}_{d_{1}}\DD_{2}\!\otimes\!\DD_{1}
\!+\f^{d_{2}}_{d_{2}}\DD_{2}\!\otimes\!\DD_{2}.\nonumber
\end{eqnarray*}\small}%
Applying $\dir_\bz$ to $[\DD,\FF_{\bz}]=-2\FF_{\bz},$ then $\DD\cdot\dir_\bz(\FF_{\bz})-\FF_{\bz}\cdot\dir_\bz(\DD)=-2\dir_\bz(\FF_{\bz})$. By \eqref{equa-d0-1}, $\FF_{\bz}\cdot\dir_\bz(\DD)=0$. Thus $\DD\cdot\dir_\bz(\FF_{\bz})=-2\dir_\bz(\FF_{\bz})$, from which, one can deduce
\begin{eqnarray*}
\f^{ee}_{\bm}\!\!\!&=&\!\!\f^{ef}_{\bm}=\f^{eg}_{\bk}=\f^{eh}_{\bk}
=\f^{fe}_{\bm}=\f^{ff}_{\bm}
=\f^{ge}_{\bk}=\f^{he}_{\bk}=\f^{gg}_{\bk}\nonumber\\
\!\!\!&=&\!\!\f^{gh}_{\bk}=\f^{hg}_{\bk}=\f^{hh}_{\bk}=\f^e_d=\f^e_{d_1}=\f^e_{d_2}
=\f^d_e=\f^{d_1}_e=\f^{d_2}_e\nonumber\\
\!\!\!&=&\!\!\f^d_d=\f^d_{d_1}=\f^{d_1}_d=\f^{d_1}_{d_1}
=\f^{d_2}_d=\f^{d_2}_{d_1}=\d^d_{d_2}=\d^{d_1}_{d_2}=\d^{d_2}_{d_2}=0.
\end{eqnarray*}
Applying $\dir_\bz$ to $[\EE_{\bz},\FF_{\bz}]=\DD$ and using \eqref{equa-e0}, we obtain $\EE_{\bz}\cdot\dir_\bz(\FF_{\bz})=\dir_\bz(\DD)$, from which and the above identities,  it is no wonder that
\begin{eqnarray*}
&&\d^{ef}_{\bz}+2\f^d_f=\d^{fe}_{\bz}+2\f^f_d=\d^d_d-\f^f_d-\f^d_f=0,\\
&&\d^{d_1}_{d_1}=\d^{d_1}_{d_2}=\d^{d_2}_{d_1}=\d^{d_2}_{d_2}
=\f^f_{d_1}=\f^f_{d_2}=\f^{d_1}_f=\f^{d_2}_f=0,\nonumber\\
&&\d^{ef}_{\bk}=\f^{hf}_{\bk}-\f^{gh}_{\bk},\ \ \d^{fe}_{\bk}=\f^{fh}_{\bk}-\f^{fg}_{\bk},\ \
\d^{gg}_{\bk}=\f^{fg}_{\bk}+\f^{gf}_{\bk},\nonumber\\
&&\d^{gh}_{\bk}=\f^{fh}_{\bk}-\f^{gf}_{\bk},\ \
\d^{hg}_{\bk}=\f^{hf}_{\bk}-\f^{fg}_{\bk},\ \
\d^{hh}_{\bk}=-\f^{hf}_{\bk}-\f^{fh}_{\bk}.\nonumber\\
\end{eqnarray*}
Using these identities and \eqref{equa-d0-1}, one deduces that
$
\dir_\bz(\DD)=\dir_\bz(\FF_{\bz})=0.
$ 
By the fact that $\WL_{\bz}=\C\EE_{\bz}\oplus\C\FF_{\bz}\oplus\C\DD\oplus\C\DD_1\oplus\C\DD_2$, we obtain the lemma.
\end{proof}

\begin{rema}\label{remak-1}\rm
We always use the convention that if an undefined symbol technically appears in an expression, we always treat it as zero; for instance, $\GG_{0,0}=\HH_{0,0}=0$.
\end{rema}

\begin{lemm}\label{lemm-L+}
By replacing $\dir_\bz$ by $\dir_\bz-u_{\rm inn}$ for some $u_{\rm inn}\in
\VV_\bz$, one can suppose $\dir_\bz(\WL_{m,n})=0$ for $m,n\in\Z^+$.
\end{lemm}
\begin{proof}Using Lemma \ref{lemm-L0} and the fact that $[d,\GG_{0,1}]=0$, we deduce $\DD\cdot\dir_\bz(\GG_{0,1})=0$. Thus we can write  (cf.~statement after \eqref{first-})
{\small\begin{eqnarray}\label{equa-G01}\!\!\!\!\!\!\!\!\!\!\dir_\bz(\GG_{0,1})=\!\!
\!\!\!\!\!\!\!\!\!&&\mbox{$\sum\limits_{(m,n)\in\BZ}$}\big(\g^{ef}_{m,n}\EE_{m,n}
\otimes\FF_{-m,1-n}+\g^{fe}_{m,n}\FF_{m,n}\otimes\EE_{-m,1-n}
+\g^{gg}_{mn}\GG_{m,n}\otimes\GG_{-m,1-n}
\nonumber\\[-10pt]
\!\!\!\!\!\!\!\!\!&&\ \ \ \ \ \ \ \ \ \ \ +\g^{gh}_{m,n}\GG_{m,n}\otimes\HH_{-m,1-n}
+\g^{hg}_{m,n}\HH_{m,n}\otimes\GG_{-m,1-n}
+\g^{hh}_{m,n}\HH_{m,n}\otimes\HH_{-m,1-n}\big)\nonumber\\
\!\!\!\!\!\!\!\!\!&&+\g^g_d\GG_{0,1}\otimes\DD+\g^d_g\DD\otimes\GG_{0,1}
+\g^g_{d_{1}}\GG_{0,1}\otimes\DD_{1}+\g^{d_{1}}_g\DD_{1}\otimes\GG_{0,1}
+\g^g_{d_{2}}\GG_{0,1}\otimes\DD_{2}+\g^{d_{2}}_g\DD_{2}\otimes\GG_{0,1}\nonumber\\
\!\!\!\!\!\!\!\!\!&&+\g^h_d\HH_{0,1}{\sc\!}\otimes{\sc\!}\DD{\sc\!}+{\sc\!}\g^d_h\DD{\sc\!}\otimes{\sc\!}\HH_{0,1}
{\sc\!}+{\sc\!}\g^h_{d_{1}}\HH_{0,1}{\sc\!}\otimes{\sc\!}\DD_{1}{\sc\!}+{\sc\!}\g^{d_{1}}_h\DD_{1}{\sc\!}\otimes{\sc\!}\HH_{0,1}
{\sc\!}+{\sc\!}\g^h_{d_{2}}\HH_{0,1}{\sc\!}\otimes{\sc\!}\DD_{2}{\sc\!}+{\sc\!}\g^{d_{2}}_h\DD_{2}{\sc\!}\otimes{\sc\!}\HH_{0,1}.
\end{eqnarray}}%
Note that
\begin{eqnarray*}
&&\GG_{0,1}\cdot( \DD_2\otimes\DD_2)=-\DD_2\otimes\GG_{0,1}-\GG_{0,1}\otimes\DD_2,\\
&&\GG_{0,1}\cdot( \DD_1\otimes\DD_2)=-\DD_1\otimes\GG_{0,1},\ \ \GG_{0,1}\cdot( \DD_2\otimes\DD_1)=-\GG_{0,1}\otimes\DD_1.
\end{eqnarray*}
Thus by replacing $\dir_\bz$ by $\dir_\bz-u_{\rm inn},$ where $u$ is some combination of   $\DD_1\otimes\DD_2$, $\DD_2\otimes\DD_1$ and $\DD_2\otimes\DD_2$,
one can assume
\begin{eqnarray}\label{equa-Ginn}
\g^g_{d_1}=\g^g_{d_2}=\g^{d_1}_g=0.
\end{eqnarray}
Similarly, we can deduce $\DD\cdot\dir_\bz(\HH_{1,0})=0$, and thus
$\dir_\bz(\HH_{1,0})$ can be written as (cf.~statement after \eqref{first-})
{\small\begin{eqnarray}\label{equa-H10}\!\!\!\!\!\!\!\dir_\bz(\HH_{1,0})=
\!\!\!\!\!\!\!\!\!&&\mbox{$\sum\limits_{(m,n)\in\BZ}$}\!\big(\h^{ef}_{m,n}\EE_{m,n}
\otimes\FF_{1-m,-n}+\h^{fe}_{m,n}\FF_{m,n}\!\otimes\!\EE_{1-m,-n}
+\h^{gg}_{m,n}\GG_{m,n}\otimes\GG_{1-m,-n}\nonumber\\[-10pt]
\!\!\!\!\!\!\!\!\!&&\ \ \ \ \ \ \ \ \,+\h^{gh}_{m,n}\GG_{m,n}\otimes\HH_{1-m,-n}+\h^{hg}_{m,n}\HH_{m,n}\otimes\GG_{1-m,-n}
+\h^{hh}_{m,n}\HH_{m,n}\otimes\HH_{1-m,-n}\big)\nonumber\\
\!\!\!\!\!\!\!\!\!&&+\h^g_d\GG_{1,0}\otimes\DD+\h^d_g\DD\otimes\GG_{1,0}
+\h^g_{d_{1}}\GG_{1,0}\otimes\DD_{1}+\h^{d_{1}}_g\DD_{1}\otimes\GG_{1,0}
+\h^g_{d_{2}}\GG_{1,0}\otimes\DD_{2}
+\h^{d_{2}}_g\DD_{2}\!\otimes\!\GG_{1,0}\nonumber\\
\!\!\!\!\!\!\!\!\!&&+\h^h_d\HH_{1,0}{\sc\!}\otimes{\sc\!}\DD{\sc\!}+{\sc\!}\h^d_h\DD{\sc\!}\otimes{\sc\!}\HH_{1,0}
{\sc\!}+{\sc\!}\h^h_{d_{1}}\HH_{1,0}{\sc\!}\otimes{\sc\!}\DD_{1}{\sc\!}+{\sc\!}\h^{d_{1}}_h\DD_{1}{\sc\!}\otimes{\sc\!}
\HH_{1,0}
{\sc\!}+{\sc\!}\h^h_{d_{2}}\HH_{1,0}{\sc\!}\otimes{\sc\!}\DD_{2}{\sc\!}+{\sc\!}\h^{d_{2}}_h\DD_{2}{\sc\!}\otimes{\sc\!}
\HH_{1,0}.
\end{eqnarray}\small}%
Since $\HH_{1,0}{\sc\!}\cdot{\sc\!}(\DD_1{\sc\!}\otimes\DD_1)$ $=-\HH_{1,0}\otimes\DD_1-\DD_1\otimes\HH_{1,0}$, by replacing $\dir_\bz$ by $\dir_\bz+u_{\rm inn}$ with $u=\h^{d_1}_h\DD_1\otimes\DD_1$,
we can suppose $\h^{d_1}_h=0$.
Applying $\dir_{\bz}$ to $[\GG_{0,1},\HH_{1,0}]=0$, we obtain $\GG_{0,1}\cdot\dir_{\bz}(\HH_{1,0})=\HH_{1,0}\cdot\dir_{\bz}(\GG_{0,1})$, which implies
\begin{eqnarray}
&&\g^h_d=\g^h_{d_1}=\g^h_{d_2}=\g^d_h=\g^{d_1}_h=\g^{d_2}_h=0,\label{equa-g01-0}\\
&&\h^g_d=\h^g_{d_2}=\h^g_{d_1}=\h^{d_1}_g=\h^{d_2}_g=\h^d_g
=\h^h_{d_2}=\h^{d_2}_h=0,\label{equa-h10-0}\\
&&(1-q^{1-n})\g^{gh}_{m,n}=(q^m-1)\h^{gh}_{m,n-1},\ \ \, (q^{1-m}-1)\h^{hg}_{m,n}=(1-q^n)\g^{hg}_{m-1,n},\label{equa-ggh-hhg}\\
&&(1-q^n)\g^{hh}_{m-1,n}=(q^{1-n}-1)\g^{hh}_{m,n},\ \ \ \, (q^m-1)\h^{gg}_{m,n-1}=(1-q^{1-m})\h^{gg}_{m,n},\label{equa-ghh-hgg}\\
&&q^m\h^{ef}_{m,n-1}\!-\h^{ef}_{m,n}\!=\g^{ef}_{m,n}\!-\!q^n\g^{ef}_{m-1,n},\ \ q^{1-m}\h^{fe}_{m,n}\!-\h^{fe}_{m,n-1}\!=\g^{fe}_{m-1,n}\!\!
-q^{1-n}\g^{fe}_{m,n}.\label{equa-gh-ef}
\end{eqnarray}
We observe the following crucial result.
\begin{clai}\label{claim1}
We have $\g^{hh}_{m,n}=\h^{gg}_{m,n}=0$ for all $(m,n)\in\BZ$.
\end{clai}

To prove the claim, consider the first equation in \eqref{equa-ghh-hgg}: If $n=0$, then $(1-q)\g^{hh}_{m,0}=0$, which implies $\g^{hh}_{m,0}=0$ for $m\in\Z$.
If $n=1$, then $(1-q)\g^{hh}_{m-1,1}=0$, which implies $\g^{hh}_{m,1}=0$ for $m\in\Z$. Now
assume
$n\neq 0,1$. Let $\mathcal {S}_n=\{m\,|\,\g^{hh}_{m,n}\neq0\}$, which is a finite set. Assume $\mathcal {S}_n\neq\emptyset$. Set $m'={\rm min}\ \mathcal {S}_n$. Using $\g^{hh}_{m-1,n}=\frac{q^{1-n}-1}{1-q^n}\g^{hh}_{m,n}$, we obtain $\g^{hh}_{m'-1,n}\neq0$, which means $m'-1\in\mathcal {S}_n$, a contradiction with the minimality of $m'$. Thus $\mathcal {S}_n=\emptyset$ and $\g^{hh}_{m,n}=0$ for  $(m,n)\in\BZ$. Analogously, consider the second equation in \eqref{equa-ghh-hgg}: If $m=0$, then $(1-q)\h^{gg}_{0,n}=0$, which implies $\h^{gg}_{0,n}=0$ for $n\in\Z$. If $m=1$, then $(1-q)\h^{gg}_{1,n-1}=0$, which implies $\h^{gg}_{1,n}=0$ for  $n\in\Z$. If $m\neq 0,1$ is fixed, then $\mathcal {S}^g_m=\{n\,|\,\h^{gg}_{m,n}\neq0\}$ is a finite set. Assume $\mathcal {S}^g_m\neq\emptyset$ and set $n'={\rm min}\ \mathcal {S}^g_m$. Using $\h^{gg}_{m,n-1}=\frac{1-q^{1-m}}{q^m-1}\h^{gg}_{m,n}$, we have $\h^{gg}_{m,n'-1}\neq0$, which implies $n'-1\in\mathcal {S}^g_m$, a contradiction with the minimality of $n'$. Thus $\mathcal {S}^g_m=\emptyset$ and $\h^{gg}_{m,n}=0$ for $(m,n)\in\BZ$. The claim is proved.\vskip6pt

Therefore, using \eqref{equa-Ginn}, \eqref{equa-g01-0} and \eqref{equa-h10-0}, one can rewrite \eqref{equa-G01} and \eqref{equa-H10} as
{\small\begin{eqnarray}
\!\!\!\!\!\!\!\!\!\dir_\bz(\GG_{0,1})\!\!\!\!&=&\!\!\!\!\!\!\!\mbox{$\sum\limits_{(m,n)\in\BZ}$}\!
\big(\g^{ef}_{m,n}\EE_{m,n}\otimes\FF_{-m,1-n}
+\g^{fe}_{m,n}\FF_{m,n}\otimes\EE_{-m,1-n}
+\g^{gg}_{m,n}\GG_{m,n}\otimes\GG_{-m,1-n}\nonumber\\[-5pt]
&&\!\!\!\!\!\!\!\!\!+\g^{gh}_{m,n}\GG_{m,n}\!\otimes\!\HH_{-m,1-n}
\!+\g^{hg}_{m,n}\HH_{m,n}\!\otimes\!\GG_{-m,1-n}\big)\!+\g^g_d\GG_{0,1}\!\otimes\!\DD\!+\g^d_g\DD\!\otimes\!\GG_{0,1}
\!+\g^{d_2}_g\DD_2\!\otimes\!\GG_{0,1},\label{equa-g01}\\[8pt]
\!\!\!\!\!\!\!\!\!\dir_\bz(\HH_{1,0})\!\!\!\!&=&\!\!\!\!\!\!\!\mbox{$\sum\limits_{(m,n)\in\BZ}$}\!
\big(\h^{ef}_{m,n}\EE_{m,n}\otimes\FF_{1-m,-n}
+\h^{fe}_{m,n}\FF_{m,n}\otimes\EE_{1-m,-n}
+\h^{gh}_{m,n}\GG_{m,n}\otimes\!\HH_{1-m,-n}\nonumber\\[-5pt]
&&\!\!\!\!\!\!\!\!\!+\h^{hg}_{m,n}\HH_{m,n}\!\otimes\!\GG_{1-m,-n}\!
+\h^{hh}_{m,n}\HH_{m,n}\!\otimes\!\HH_{1-m,-n}\big)
\!\!+\h^h_d\HH_{1,0}\!\otimes\!\DD\!\!+\h^h_{d_{1}}\HH_{1,0}\!\otimes\!\DD_{1}
\!\!+\h^d_h\DD\!\otimes\!\HH_{1,0},\label{equa-h10}
\end{eqnarray}
\small}

\noindent where the coefficients satisfy \eqref{equa-ggh-hhg} and \eqref{equa-gh-ef}. Applying $\dir_{\bz}$ to $[\EE_{\bz},\GG_{0,1}]=-\EE_{0,1}$, we can write $\dir_{\bz}(\EE_{0,1})$ as
\begin{eqnarray}\!\!\!\!\dir_{\bz}(\EE_{0,1})=
&&\!\!\!\!\!\!\!\!\!\!
\mbox{$\sum\limits_{{(m,n)\in\BZ^*\setminus(0,1)}}$}\!
\big((\g^{gg}_{m,n}-\g^{ef}_{m,n}
-\g^{hg}_{m,n})\EE_{m,n}\otimes\GG_{-m,1-n}
+(\g^{ef}_{m,n}+\g^{gh}_{m,n})\EE_{m,n}\otimes\HH_{-m,1-n}\nonumber\\[-5pt]
&&\ \ \ \ +(\g^{gg}_{m,n}-\g^{fe}_{m,n}-\g^{gh}_{m,n})\GG_{m,n}\otimes\EE_{-m,1-n}
+(\g^{fe}_{m,n}+\g^{hg}_{m,n})\HH_{m,n}\otimes\EE_{-m,1-n}\big)\nonumber\\
&&+(2\g^d_g\!-\g^{ef}_{\bz})\EE_{\bz}\otimes\GG_{0,1}
+(2\g^g_d\!-\g^{fe}_{\bz})\GG_{0,1}\otimes\EE_{\bz}\!+\g^{ef}_{\bz}\EE_{\bz}\otimes\HH_{0,1}
\!+\g^{fe}_{0,1}\HH_{0,1}\otimes\EE_{\bz}\nonumber\\
&&+(\g^g_d-\g^{ef}_{0,1})\EE_{0,1}\otimes\DD
+(\g^d_g-\g^{fe}_{\bz})\DD\otimes\EE_{0,1}+\g^{d_2}_g\DD_2\otimes\EE_{0,1}.\label{equa-E01}
\end{eqnarray}
Using Lemma \ref{lemm-L0} and $[\EE_{\bz},\EE_{0,1}]=0$, one deduces $\EE_{\bz}\cdot\dir_{\bz}(\EE_{0,1})=0$. Thus
\begin{eqnarray}
&&\g^{ef}_{\bz}+\g^{fe}_{\bz}=2\g^d_g,\ \ \ \ \g^{fe}_{0,1}+\g^{ef}_{0,1}=2\g^g_d,\label{equa-gdgh}\\
&&\g^{ef}_{m,n}+\g^{fe}_{m,n}=\g^{gg}_{m,n}-\g^{hg}_{m,n}-\g^{gh}_{m,n}\mbox{ \ if \ } (m,n)\neq (0,0),(0,1).\label{equa-gefgh}
\end{eqnarray}
Applying $\dir_{\bz}$ to $[\EE_{0,1},\FF_{\bz}]=\GG_{0,1}-\HH_{0,1}$, by Lemma \ref{lemm-L0}, we obtain $$\dir_{\bz}(\HH_{0,1})=\dir_{\bz}(\GG_{0,1})+\FF_{\bz}\cdot\dir_{\bz}(\EE_{0,1}).$$
Using \eqref{equa-E01}, one can then write $\dir_{\bz}(\HH_{0,1})$ as
{\small\begin{eqnarray*}
\!\!\!&&\!\!\!\!\!\!\!\!\mbox{$\sum\limits_{{(m,n)\in\BZ^*\setminus(0,1)}}$}\!\!\big((\g^{gg}_{m,n}\!-\g^{hg}_{m,n}
\!-\g^{ef}_{m,n}\!-\g^{gh}_{m,n})\EE_{m,n}\!\otimes\!\FF_{-m,1\!-\!n}\!
\!+\!(\g^{gg}_{m,n}\!-\g^{hg}_{m,n}\!-\g^{fe}_{m,n}\!-\g^{gh}_{m,n})\FF_{m,n}\!\otimes\!\EE_{-m,1-n}\nonumber\\[-5pt]
\!\!\!&&\ \ +(\g^{gg}_{m,n}-\g^{gh}_{m,n}-\g^{ef}_{m,n}-\g^{fe}_{m,n})\GG_{m,n}\!\otimes\!\HH_{-m,1-n}
+\!(\g^{gg}_{m,n}-\g^{hg}_{m,n}-\g^{ef}_{m,n}-\g^{fe}_{m,n})\HH_{m,n}\!\otimes\!\GG_{-m,1-n}\nonumber\\
\!\!\!&&\ \ +(\g^{ef}_{m,n}\!\!+\g^{fe}_{m,n}\!\!-\g^{gg}_{m,n}\!\!+\g^{gh}_{m,n}\!+\g^{hg}_{m,n})\GG_{m,n}\!\otimes\!\GG\!_{-m,1\!-\!n}\!+\!(\g^{ef}_{m,n}
\!\!+\g^{fe}_{m,n}\!\!+\g^{gh}_{m,n}\!\!+\g^{hg}_{m,n})\HH_{m,n}\!\otimes\!\HH\!_{-m,1\!-\!n}\big)\\
\!\!\!&&+(2\g^d_g\!-\g^{ef}_{\bz})\EE_{\bz}\otimes\FF_{0,1}+(2\g^g_d\!-\g^{ef}_{0,1})\EE_{0,1}\otimes\FF_{\bz}
+(2\g^d_g\!-\g^{fe}_{\bz})\FF_{\bz}\otimes\EE_{0,1}+(2\g^g_d\!-\g^{ef}_{0,1})\FF_{0,1}\otimes\EE_{\bz}\\
\!\!\!&&+(\g^{ef}_{\bz}+\g^{fe}_{\bz}-2\g^d_g)\DD\otimes\GG_{0,1}
+(\g^{fe}_{0,1}+\g^{ef}_{0,1}-2\g^g_d)\GG_{0,1}\otimes\DD
+(\g^d_g-\g^{ef}_{\bz}-\g^{fe}_{\bz})\DD\otimes\HH_{0,1}\\
\!\!\!&&+(\g^g_d-\g^{ef}_{0,1}-\g^{fe}_{0,1})\HH_{0,1}\otimes\DD
+\g^{d_2}_g\DD_2\otimes\HH_{0,1}.
\end{eqnarray*}}%
By \eqref{equa-gdgh} and \eqref{equa-gefgh}, one can rewrite $\dir_{\bz}(\HH_{0,1})$ as
\begin{eqnarray*}\dir_{\bz}(\HH_{0,1})=
\!\!\!\!\!\!&\!\!\!\!\!\!&\!\!\!\!\!\!
\mbox{$\sum\limits_{(m,n)\in\BZ}$}(\g^{fe}_{m,n}\EE_{m,n}\otimes\FF_{-m,1-n}
+\g^{ef}_{m,n}\FF_{m,n}\otimes\EE_{-m,1-n})
-\g^d_g\DD\otimes\HH_{0,1}-\g^g_d\HH_{0,1}\otimes\DD\\[-5pt]
\!\!\!\!\!\!&+\!\!&\!\!\!\!\!\!\!\!\!\!\!\!\mbox{$\sum\limits_{(m,n)\in\BZ^*\setminus(0,1)}$}\!\!
(\g^{hg}_{m,n}\GG_{m,n}\!\otimes\!\HH\!\!_{-m,1-n}\!+\g^{gh}_{m,n}\HH_{m,n}\!\otimes\!\GG\!\!_{-m,1-n}
\!+\g^{gg}_{m,n}\HH_{m,n}\!\otimes\!\HH\!\!_{-m,1-n})\!+\g^{d_2}_g\DD_2\!\otimes\!\HH_{0,1}.
\end{eqnarray*}
Applying $\dir_{\bz}$ to $[\GG_{0,1},\HH_{0,1}]=0$ and using \eqref{equa-g01}, we have
\begin{eqnarray}
&&\g^{d_2}_g=0,\ \ (q^m-1)\g^{hg}_{m,n-1}=(q^{-m}-1)\g^{gh}_{m,n},\label{equa-ghg}\\
&&q^m\g^{fe}_{m,n-1}-\g^{fe}_{m,n}=q^{-m}\g^{ef}_{m,n}-\g^{ef}_{m,n-1}.\label{equa-gfe-gef}
\end{eqnarray}
Applying $\dir_{\bz}$ to $[\EE_{\bz},\HH_{1,0}]=\EE_{1,0}$, then
{\small\begin{eqnarray}
\dir_{\bz}(\EE_{1,0})&\!\!\!=\ \ &\!\!\!\!\!\!\!\!\!
\mbox{$\sum\limits_{(m,n)\in\BZ^*\setminus(1,0)}$}\!\!\big((\h^{ef}_{m,n}\!\!+\h^{hg}_{m,n})\EE_{m,n}\!\otimes\!\GG\!_{1-m,-n}
\!+\!(\h^{hh}_{m,n}\!\!-\h^{ef}_{m,n}\!\!-\h^{gh}_{m,n})\EE_{m,n}\!\otimes\!\HH\!_{1-m,-n}\nonumber\\[-6pt]
&&\!\!\!\!\!\!\!\!\!+(\h^{fe}_{m,n}+\h^{gh}_{m,n})\GG_{m,n}\otimes\EE_{1-m,-n}
+(\h^{hh}_{m,n}-\h^{fe}_{m,n}-\h^{hg}_{m,n})\HH_{m,n}\otimes\EE_{1-m,-n}\big)\nonumber\\
&&\!\!\!\!\!\!\!\!\!-\h^{ef}_{\bz}\EE_{\bz}\!\otimes\!\GG_{1,0}
\!+\h^{fe}_{1,0}\GG_{1,0}\!\otimes\!\EE_{\bz}\!-(\h^{ef}_{\bz}\!+2\h^d_h)\EE_{\bz}\!\otimes\!\HH_{1,0}
\!-(\h^{fe}_{1,0}\!+2\h^h_d)\HH_{1,0}\!\otimes\!\EE_{\bz}\nonumber\\
&&\!\!\!\!\!\!\!\!\!+(\h^h_d+\h^{ef}_{1,0})\EE_{1,0}\otimes\DD
+(\h^d_h+\h^{fe}_{\bz})\DD\otimes\EE_{1,0}+\h^h_{d_1}\EE_{1,0}\otimes\DD_1.
\label{equa-E10}
\end{eqnarray}}%
Using $[\EE_{\bz},\EE_{1,0}]=0$ and Lemma \ref{lemm-L0}, one has $\EE_{\bz}\cdot\dir_{\bz}(\EE_{1,0})=0$. Thus
\begin{eqnarray}
&&\h^{ef}_{\bz}+\h^{fe}_{\bz}+2\h^d_h=0,\ \ \ \h^{fe}_{1,0}+\h^{ef}_{1,0}+2\h^h_d=0,\label{equa-hefd}\\
&&\h^{ef}_{m,n}+\h^{fe}_{m,n}=\h^{hh}_{m,n}-\h^{hg}_{m,n}-\h^{gh}_{m,n}\mbox{ \ for \ } (m,n)\neq (0,0),(1,0).\label{equa-hefgh}
\end{eqnarray}
Applying $\dir_{\bz}$ to $[\EE_{1,0},\FF_{\bz}]=\GG_{1,0}-\HH_{1,0}$ and using Lemma \ref{lemm-L0}, there is no wonder that $\dir_{\bz}(\GG_{1,0})=\dir_{\bz}(\HH_{1,0})-\FF_{\bz}\cdot\dir_{\bz}(\EE_{1,0})$. Then by \eqref{equa-E10}, one can write $\dir_{\bz}(\GG_{1,0})$ as
\begin{eqnarray*}
\!\!\!\!\!\!\!\!&&\!\!\!\!\!\!\!\!\mbox{$\sum\limits_{{(m,n)\in\BZ^*\setminus(1,0)}}$}
\!\!\big((\h^{hh}_{m,n}\!\!-\h^{hg}_{m,n}
\!\!-\h^{ef}_{m,n}\!\!-\h^{gh}_{m,n})\EE_{m,n}\!\otimes\!\FF\!_{1-m,-n}\!\!\!+\!(\h^{hh}_{m,n}\!\!-\h^{hg}_{m,n}
\!\!-\h^{fe}_{m,n}\!\!-\h^{gh}_{m,n})\FF\!_{m,n}\!\otimes\!\EE_{1-m,-n}\\[-7pt]
\!\!\!\!&&+(\h^{ef}_{m,n}\!+\h^{fe}_{m,n}
\!+\h^{hg}_{m,n}\!+\h^{gh}_{m,n})\GG_{m,n}\!\otimes\!\GG_{1-m,-n}\!+\
(\h^{hh}_{m,n}\!-\h^{gh}_{m,n}\!-\h^{ef}_{m,n}\!-\h^{fe}_{m,n})\GG_{m,n}\!\otimes\!\HH_{1-m,-n}\\
&&+(\h^{hh}_{m,n}
\!-\h^{hg}_{m,n}\!-\h^{ef}_{m,n}\!-\h^{fe}_{m,n})\HH_{m,n}\!\otimes\!\GG_{1-m,-n}\!+\!(\h^{ef}_{m,n}
\!+\h^{fe}_{m,n}\!+\h^{gh}_{m,n}\!+\h^{hg}_{m,n})\HH_{m,n}\!\otimes\!\HH_{1-m,-n}\big)\\
&&-(2\h^d_h\!+\h^{ef}_{\bz})\EE_{\bz}\!\otimes\!\FF_{1,0}-(2\h^h_d\!+\h^{ef}_{1,0})\EE_{1,0}\!\otimes\!\FF_{\bz}
\!-\!(2\h^d_h\!+\h^{fe}_{\bz})\FF_{\bz}\!\otimes\!\EE_{1,0}\!-\!(2\h^h_d\!+\h^{ef}_{1,0})\FF_{1,0}\!\otimes\!\EE_{\bz}\\
&&+(\h^{ef}_{\bz}\!+\h^{fe}_{\bz}-2\h^d_h)\DD\otimes\HH_{1,0}
\!+\!(\h^{fe}_{1,0}+\h^{ef}_{1,0}\!-2\h^h_d)\HH_{1,0}\otimes\DD
+(\h^d_h-\h^{ef}_{\bz}-\h^{fe}_{\bz})\DD\otimes\GG_{1,0}\\
&&+(\h^h_d-\h^{ef}_{1,0}-\h^{fe}_{1,0})\GG_{1,0}\otimes\DD
+\h^h_{d_1}\GG_{1,0}\otimes\DD_1.
\end{eqnarray*}
By \eqref{equa-hefd} and \eqref{equa-hefgh}, we see that $\dir_{\bz}(\GG_{1,0})+\h^d_h\DD\otimes\GG_{1,0}+\h^h_d\GG_{1,0}\otimes\DD
-\h^h_{d_1}\GG_{1,0}\otimes\DD_1$ is equal to
\begin{eqnarray}
&\!\!&\!\!\!\!\!\!\!\!\!\!\!
\mbox{$\sum\limits_{(m,n)\in\BZ}$}\!\!(\h^{fe}_{m,n}\EE_{m,n}\otimes\FF_{1-m,-n}
+\h^{ef}_{m,n}\FF_{m,n}\otimes\EE_{1-m,-n})\nonumber\\[-8pt]
&&\!\!\!\!\!\!\!\!\!\!\!
+\!\!\!\!\!\mbox{$\sum\limits_{{(m,n)\in\BZ^*\setminus(1,0)}}$}\!\!(\h^{hh}_{m,n}\GG_{m,n}\!\otimes\!\GG_{1-m,-n}\!\!
+\h^{hg}_{m,n}\GG_{m,n}\!\otimes\!\HH_{1-m,-n}\!\!
+\h^{gh}_{m,n}\HH_{m,n}\!\otimes\!\GG_{1-m,-n}).\label{equa-G10}
\end{eqnarray}
Applying $\dir_{\bz}$ to $[\GG_{1,0},\HH_{1,0}]=0$, by \eqref{equa-h10} and \eqref{equa-G10}, we have
\begin{eqnarray}
&&\h^{d_1}_h=0,\ \ \ \ \ \h^{hg}_{m,n}=-q^n\h^{gh}_{m-1,n},\label{equa-hhg}\\
&&q^n\h^{fe}_{m-1,n}-\h^{fe}_{m,n}=q^{-n}\h^{ef}_{m,n}-\h^{ef}_{m-1,n}.\label{equa-hfe-hef}
\end{eqnarray}

\label{page12}We now consider the second equation of \eqref{equa-ghg}. We observe an important fact that,
as stated in the introduction, in order to be able to determine $\dir_0$ (or in order to simplify some systems of equations such as \eqref{equa-ghg}), some extra inner derivations must be subtracted from $\dir_0$. We shall do this in two steps (cf.~\eqref{inner-u} and \eqref{inner-v}). Firstly, we take
\begin{equation}\label{inner-u}
\mbox{$u=\sum\limits_{m\neq 0}\sum\limits_{n}$}\frac{\g^{gh}_{m,n+1}}{q^m-1}(\GG_{m,n}+\HH_{m,n})\otimes(\GG_{-m,-n}+\HH_{-m,-n}).\end{equation}
Then we have by using  \eqref{equa-ghg} and relations \eqref{re-ss},
\begin{eqnarray*}
\GG_{0,1}\cdot u\!\!\!&=&\!\!\!
\mbox{$\sum\limits_{m\neq0}\sum\limits_{n}$}\frac{\g^{gh}_{m,n+1}}{q^m-1}
\big((q^m-1)\GG_{m,1+n}\otimes(\GG_{-m,-n}+\HH_{-m,-n})
\\[-10pt]
&&\ \ \ \ \ \ \ \ +(\GG_{m,n}+\HH_{m,n})\otimes(q^{-m}-1)\GG_{-m,1-n}\big)\\
&=&\mbox{$\sum\limits_{m\neq0}\sum\limits_{n}$}(\g^{gh}_{m,n+1}\GG_{m,1+n}\otimes\HH_{-m,-n}
-q^m\g^{gh}_{m,n+1}\HH_{m,n}\otimes\GG_{-m,1-n}\\[-10pt]
&&\ \ \ +\g^{gh}_{m,n+1}\GG_{m,1+n}\otimes\GG_{-m,-n}
-q^{-m}\g^{gh}_{m,n+1}\GG_{m,n+1}\otimes\GG_{-m,1-n})\\
&=&\mbox{$\sum\limits_{m\neq0}\sum\limits_{n}$}(\g^{gh}_{m,n+1}\GG_{m,1+n}\otimes\HH_{-m,-n}
+\g^{hg}_{m,n}\HH_{m,n}\otimes\GG_{-m,1-n}\\[-10pt]
&&\ \ \ +\g^{gh}_{m,n+1}\GG_{m,1+n}\otimes\GG_{-m,-n}
+q^{-m}\g^{gh}_{m,n+1}\GG_{m,n+1}\otimes\GG_{-m,1-n}).
\end{eqnarray*}
Thus, by
replacing $\dir_\bz$ by $\dir_\bz-u_{\rm inn},$ and by \eqref{equa-g01}, we obtain $\g^{hg}_{m,n}=\g^{gh}_{m,n}=0$ for $m\in\Z^*,$ $n\in\Z$. By \eqref{equa-ggh-hhg}, $\g^{gh}_{0,n}=\g^{hg}_{0,n}=0$ for $n\neq 0,1$. Hence, $\g^{gh}_{m,n}=\g^{hg}_{m,n}=0$ for $(m,n)\neq(0,0),(0,1)$. By \eqref{equa-ggh-hhg}, $\h^{gh}_{m,n}=0$ for $m\in\Z^*$, $n\in\Z$, and $\h^{hg}_{m,n}=0$ for $m,n\in\Z$ with $m\ne1$.

Secondly, introduce
\begin{equation}\label{inner-v}
v=\mbox{$\sum\limits_{n\neq0}$}\frac{\h^{gh}_{0,n}}{1-q^n}(\GG_{0,n}+\HH_{0,n})\otimes
(\GG_{0,-n}+\HH_{0,-n}).\end{equation} Using \eqref{equa-hhg} and relations \eqref{re-ss}, we have
\begin{eqnarray}\label{SMD--}
\GG_{1,0}\cdot v\!\!\!&=&\!\!\!0,\\
\HH_{1,0}\cdot v\!\!\!&=&\!\!\!
\mbox{$\sum\limits_{n\neq0}$}\frac{\h^{gh}_{0,n}}{1-q^n}\big((1-q^n)\HH_{1,n}\otimes(\GG_{0,-n}+\HH_{0,-n})
+(\GG_{0,n}+\HH_{0,n})\otimes(1-q^n)\HH_{1,-n}\big)\nonumber\\[-6pt]
\!\!\!&=&\!\!\!\mbox{$\sum\limits_{n\neq0}$}(\h^{hg}_{1,n}\HH_{1,n}\!\otimes\!\GG_{0,-n}\!\!-\!q^{-n}\h^{hg}_{1,n}\GG_{0,n}\!\otimes\!\HH_{1,-n}
\!\!+\h^{hg}_{1,n}\HH_{1,n}\!\otimes\!\HH_{0,-n}\!\!+\!q^{-n}\h^{hg}_{1,n}\HH_{0,n}\!\otimes\!\HH_{1,-n})\nonumber\\[-6pt]
\!\!\!&=&\!\!\!\mbox{$\sum\limits_{n\neq0}$}(\h^{hg}_{1,n}\HH_{1,n}\otimes\GG_{0,-n}+\h^{gh}_{0,n}\GG_{0,n}\otimes\HH_{1,-n}
+\h^{hh}_{1,n}\HH_{1,n}\otimes\HH_{0,-n}+\h^{hh}_{0,n}\HH_{0,n}\otimes\HH_{1,-n}).\nonumber
\end{eqnarray}
Thus, if we replace $\dir_\bz$ further by $\dir_\bz-v_{\rm inn}$ (note from \eqref{SMD--} that the replacement does not affect $\dir_0(\GG_{1,0})$), we have $\h^{hg}_{1,n}=\h^{gh}_{0,n}=0$ for $n\in\Z^*$. Thus $\g^{gh}_{m,n}=\g^{hg}_{m,n}=\h^{gh}_{m,n}=\h^{hg}_{m,n}=0$ for $(m,n)\neq(0,0),(1,0)$.

Another crucial result we observed is the following.
\begin{clai}\label{claim2}We have $\g^{gg}_{m,n}= \h^{hh}_{m,n}=0$ for $(m,n)\neq(0,0),(1,0)$.
\end{clai}

To prove the claim, first we consider elements $\g^{ef}_{m',n}$ and set $\mathcal {S}_{m'}=\{n\,|\,\g^{ef}_{m',n}\neq0\}$ for  $m'\in\Z$.
From \eqref{equa-gdgh} and \eqref{equa-gefgh}, we can see that the claim holds if $\mathcal {S}_{m'}=\emptyset$ for all $m'\in\Z$. Thus
assume $\mathcal {S}_{m'}\ne\emptyset$ for some fixed $m'\in\Z$ and let $n_0={\rm min}\,\mathcal {S}_{m'}$. For any $n\in\Z$, if $n<n_0$, then \eqref{equa-gfe-gef} gives $q^{m'}\g^{fe}_{m',n-1}=\g^{fe}_{m',n}$, which implies $\g^{fe}_{m',n}=0$ due to the fact that there are only finite many nonzero $\g^{fe}_{m,n}$'s. If $n=n_0$, then $\g^{fe}_{m',n_0}=-q^{-m'}\g^{ef}_{m',n_0}$. By induction on $n$, we have $q^{m'}\g^{fe}_{m',n}=-\g^{ef}_{m',n}$ for $n\in\Z$ and
\begin{eqnarray}
&&q^m\g^{fe}_{m,n}=-\g^{ef}_{m,n}\ \ \ \ \ \ \ {\rm for \ }\ (m,n)\in\BZ,\nonumber\\
&&\g^{gg}_{m,n}=(q^m-1)\g^{ef}_{m,n}\ \ {\rm for \ }\  (m,n)\neq(0,0),(0,1).\label{equa-ggg-ef}
\end{eqnarray}
Using \eqref{equa-hhg} and \eqref{equa-hfe-hef}, the following can be obtained by the similar methods:
\begin{eqnarray}
&&\h^{ef}_{m,n}=-q^n\h^{fe}_{m,n}\ \ \ \ \ \ \ \,{\rm for\ all}\ (m,n)\in\BZ,\nonumber\\
&&\h^{hh}_{m,n}=(1-q^n)\h^{fe}_{m,n}\ \ {\rm for}\ (m,n)\neq(0,0),(1,0).\label{equa-hhh-fe}
\end{eqnarray}
Then \eqref{equa-E01} can be rewritten as
\begin{eqnarray}
\dir_{\bz}(\EE_{0,1})=\!\!\!\!\!\!\!\!&&\!\!
\mbox{$\sum\limits_{{(m,n)\in\BZ\setminus(0,1)}}$}\!(\g^{fe}_{m,n}\EE_{m,n}\otimes\GG_{-m,1-n}
+\g^{ef}_{m,n}\EE_{m,n}\otimes\HH_{-m,1-n})\nonumber\\[-6pt]
&+\!\!&
\mbox{$\sum\limits_{(m,n)\in\BZ^*}$}(\g^{ef}_{m,n}\GG_{m,n}\otimes\EE_{-m,1-n}
+\g^{fe}_{m,n}\HH_{m,n}\otimes\EE_{-m,1-n})\nonumber\\[-6pt]
&+\!\!&\!\!(\g^g_d-\g^{ef}_{0,1})\EE_{0,1}\otimes\DD
+(\g^d_g-\g^{fe}_{\bz})\DD\otimes\EE_{0,1}.\label{equa-e01}
\end{eqnarray}
The identity $[\dir_{\bz}(\FF_{0,0}),\dir_{\bz}(\GG_{0,1})]=\dir_{\bz}(\FF_{0,1})$ gives
\begin{eqnarray}
\dir_{\bz}(\FF_{0,1})=\!\!\!&&\!\!\mbox{$\sum\limits_{{(m,n)\in\BZ\setminus(0,1)}}$}
\!(\g^{ef}_{m,n}\FF_{m,n}\otimes\GG_{-m,1-n}
+\g^{fe}_{m,n}\FF_{m,n}\otimes\HH_{-m,1-n})\nonumber\\[-6pt]
&+\!\!&\mbox{$\sum\limits_{(m,n)\in\BZ^*}$}\!(\g^{fe}_{m,n}\GG_{m,n}\otimes\FF_{-m,1-n}
+\g^{ef}_{m,n}\HH_{m,n}\otimes\FF_{-m,1-n})\nonumber\\[-6pt]
&+\!\!&\!\!(\g^g_d-\g^{fe}_{0,1})\FF_{0,1}\otimes\DD
+(\g^d_g-\g^{ef}_{\bz})\DD\otimes\FF_{0,1}.\label{equa-F01}
\end{eqnarray}
Applying $\dir_{\bz}$ to $[\GG_{0,1},[\EE_{0,1},\FF_{0,1}]]=0$, one sees that $$\GG_{0,1}\cdot\big(\EE_{0,1}\cdot\dir_{\bz}(\FF_{0,1})-\FF_{0,1}\cdot\dir_{\bz}(\EE_{0,1})\big)
=(\GG_{0,2}-\HH_{0,2})\cdot\dir_{\bz}(\GG_{0,1}).$$
By \eqref{equa-g01}, \eqref{equa-e01} and \eqref{equa-F01}, and by comparing the coefficients of $\GG_{m,n}\otimes\GG_{-m,3-n}$, we obtain
\begin{eqnarray}
q^{-2m}\g^{gg}_{m,n}=q^m\g^{gg}_{m,n-2}+(q^{-m}-1)\g^{gg}_{m,n-1}.\label{equa-ggg}
\end{eqnarray}
For any fixed $m'$, the set $\mathcal {S}^g_{m'}=\{n\,|\,\g^{gg}_{m',n}\neq0\}$ is finite. Assume $\mathcal {S}^g_{m'}\neq\emptyset$ and set $n^g_0={\rm min\,}\mathcal {S}^g_{m'}$. Using \eqref{equa-ggg} and arguing as above, one sees that $\g^{gg}_{m,n}=0$ for $(m,n)\neq(0,0),(0,1)$.

Similarly, \eqref{equa-E10} can be rewritten as
\begin{eqnarray}
\dir_{\bz}(\EE_{1,0})=\!\!\!\!\!\!\!\!&&\mbox{$\sum\limits_{(m,n)\in\BZ\setminus(1,0)}$}
\!(\h^{fe}_{m,n}\EE_{m,n}\otimes\HH_{1-m,-n}
+\h^{ef}_{m,n}\EE_{m,n}\otimes\GG_{1-m,-n})\nonumber\\[-6pt]
&+\!\!&\mbox{$\sum\limits_{(m,n)\in\BZ^*}$}\!
(\h^{ef}_{m,n}\HH_{m,n}\otimes\EE_{1-m,-n}
+\h^{fe}_{m,n}\GG_{m,n}\otimes\EE_{1-m,-n})\nonumber\\[-6pt]
&+\!\!&\!\!(\h^h_d-\h^{ef}_{1,0})\EE_{1,0}\otimes\DD
+(\h^d_h-\h^{fe}_{\bz})\DD\otimes\EE_{1,0}.\label{equa-e10}
\end{eqnarray}
We can write $\dir_{\bz}(\FF_{1,0})$ as follows due to the identity $[\dir_{\bz}(\HH_{1,0}),\dir_{\bz}(\FF_{0,0})]=\dir_{\bz}(\FF_{1,0})$:
\begin{eqnarray}
\dir_{\bz}(\FF_{1,0})=\!\!\!&&\mbox{$\sum\limits_{(m,n)\in\BZ\setminus(1,0)}$}\!
(\h^{ef}_{m,n}\FF_{m,n}\otimes\HH_{1-m,-n}
+\h^{fe}_{m,n}\FF_{m,n}\otimes\GG_{1-m,-n})\nonumber\\[-6pt]
&+\!\!&\mbox{$\sum\limits_{(m,n)\in\BZ^*}$}\!(\h^{fe}_{m,n}\HH_{m,n}\otimes\FF_{1-m,-n}
+\h^{ef}_{m,n}\GG_{m,n}\otimes\FF_{1-m,-n})\nonumber\\[-6pt]
&+\!\!&\!\!(\h^h_d-\h^{fe}_{1,0})\FF_{1,0}\otimes\DD
+(\h^d_h-\h^{ef}_{\bz})\DD\otimes\FF_{0,1}.\label{equa-F10}
\end{eqnarray}
Applying $\dir_{\bz}$ to $[\HH_{1,0},[\EE_{1,0},\FF_{1,0}]]=0$, we have $$\HH_{1,0}\cdot(\EE_{1,0}\cdot\dir_{\bz}(\FF_{1,0})-\FF_{1,0}\cdot\dir_{\bz}(\EE_{1,0}))
=(\GG_{2,0}-\HH_{2,0})\cdot\dir_{\bz}(\HH_{1,0}).$$
Combining \eqref{equa-h10}, \eqref{equa-e10} and \eqref{equa-F10}, and comparing the coefficients of $\HH_{m,n}\otimes\HH_{3-m,-n}$, one can deduce that
\begin{eqnarray}\label{equa-ggg-h}
q^{-2n}\h^{hh}_{m,n}=q^n\h^{hh}_{m-2,n}+(q^{-n}-1)\h^{hh}_{m-1,n}.
\end{eqnarray}
For any fixed $n'$, as before, the set $\mathcal {S}^h_{n'}=\{m\,|\,\h^{hh}_{m,n'}\neq0\}$ is finite.
Assume $\mathcal {S}^h_{n'}\neq\emptyset$ and set $m_0^h={\rm min\,}\mathcal {S}^h_{n'}$. Using \eqref{equa-ggg-h} and arguing as above, we can obtain $\h^{hh}_{m,n}=0$ for $(m,n)\neq(0,0),(1,0)$.
This proves Claim \ref{claim2}.\vskip6pt

Take $m=0$ and  $n=0$ respectively in the first equation of \eqref{equa-gh-ef}, then
\begin{eqnarray*}
&&\h^{ef}_{0,n-1}-\h^{ef}_{0,n}=\g^{ef}_{0,n}-q^n\g^{ef}_{-1,n}=\g^{ef}_{0,n},\\
&&q^m\h^{ef}_{m,-1}-\h^{ef}_{m,0}=-\g^{ef}_{m-1,0}+\g^{ef}_{m,0},
\end{eqnarray*}
which together with \eqref{equa-ggg-ef} and \eqref{equa-hhh-fe} imply $\g^{ef}_{m,n}=\g^{fe}_{m,n}=0$ for $(m,n)\neq(0,0),(0,1)$ and $\h^{ef}_{m,n}=\h^{fe}_{m,n}=0$ for $(m,n)\neq(0,0),(1,0)$. Furthermore, we also obtain $\h^{ef}_{\bz}=-\g^{ef}_{\bz}=\g^{ef}_{0,1}=-\h^{ef}_{1,0}=-\h^{fe}_{\bz}
=\h^{fe}_{1,0}=\g^{fe}_{\bz}=-\g^{fe}_{0,1}$. Then \eqref{equa-hefd} and \eqref{equa-gdgh} imply $\g^d_g=\g^g_d=\h^h_d=\h^d_h=0$.

By now, \eqref{equa-g01} and \eqref{equa-h10} can be rewritten as
\begin{eqnarray*}
\dir_{\bz}(\GG_{0,1})\!\!\!&=&\!\!\g^{ef}_{\bz}(\EE_{\bz}\otimes\FF_{0,1}-\EE_{0,1}\otimes\FF_{\bz}
-\FF_{\bz}\otimes\EE_{0,1}+\FF_{0,1}\otimes\EE_{\bz}),\\
\dir_{\bz}(\HH_{1,0})\!\!\!&=&\!\!\g^{ef}_{\bz}(-\EE_{\bz}\otimes\FF_{1,0}+\EE_{1,0}\otimes\FF_{\bz}
+\FF_{\bz}\otimes\EE_{1,0}-\FF_{1,0}\otimes\EE_{\bz}).
\end{eqnarray*}
Replacing $\dir_\bz$ by $\dir_\bz-u_{\rm inn},$ where $u=-\g^{ef}_{\bz}\EE_{\bz}\otimes\FF_{\bz}+\FF_{\bz}\otimes\EE_{\bz}+\frac{1}{2}\DD\otimes\DD$,
we also have $\g^{e,f}_{\bz}=0$ by using
\begin{eqnarray*}
\GG_{0,1}\cdot u\!\!\!&=&\!\!\!\EE_{0,1}\otimes\FF_{\bz}-\EE_{\bz}\otimes\FF_{0,1}
-\FF_{0,1}\otimes\EE_{\bz}+\FF_{\bz}\otimes\EE_{0,1},\\
\HH_{1,0}\cdot u\!\!\!&=&\!\!\!-\EE_{1,0}\otimes\FF_{\bz}+\EE_{\bz}\otimes\FF_{1,0}
+\FF_{1,0}\otimes\EE_{\bz}-\FF_{\bz}\otimes\EE_{1,0}.
\end{eqnarray*}
 Thus $\dir_{\bz}({\GG_{0,1}})=\dir_{\bz}(\HH_{1,0})=0$. By the similar method, one can deduce that $\dir_{\bz}({\GG_{0,n}})=\dir_{\bz}(\HH_{n,0})=0$ for $n\in\Z^+$.
Now using Lemma \ref{lemm-L0} and the fact that $\WL_{m,n}$ for $m,n\in\Z^+$ can be generated by the set $\{\DD,\,\EE_{0,0},\,\FF_{0,0},\,\DD_{1},\,\DD_{2},\,\GG_{0,n},\,\HH_{n,0}\}$,  we complete the proof of Lemma \ref{lemm-L+}.
\end{proof}

\begin{lemm}\label{lemm-L-}
Replacing $\dir_\bz$ by $\dir_\bz-u_{\rm inn}$ for some $u\in
\VV_\bz$, one can suppose $\dir_\bz(\WL_{m,n})=0$ for $m,n\in\Z^-$.
\end{lemm}
\begin{proof}
Since $\DD\cdot\dir_\bz(\GG_{0,-1})=0$, we can write $\dir_\bz(\GG_{0,-1})$ as (cf.~statement after \eqref{first-})
\begin{eqnarray}\label{equa-G0-1}
\!\!\!\!\!\!\!\!\!\!\!\!\!\!\!&&\!\!\!\!\!\!\!\!\!\!\!\!\mbox{$\sum\limits_{(m,n)\in\BZ}$}\!\big(\G^{ef}_{m,n}\EE_{m,n}
\otimes\FF_{-m,-1-n}+\G^{fe}_{m,n}\FF_{m,n}\otimes\EE_{-m,-1-n}
+\G^{gg}_{m,n}\GG_{m,n}\otimes\GG_{-m,-1-n}\nonumber\\[-12pt]
\!\!\!\!\!\!\!\!\!\!\!\!\!\!\!&&\ \ \ +\G^{gh}_{m,n}\GG_{m,n}\otimes\HH_{-m,-1-n}
+\G^{hg}_{m,n}\HH_{m,n}\otimes\GG_{-m,-1-n}
+\G^{hh}_{m,n}\HH_{m,n}\otimes\HH_{-m,-1-n}\big)\nonumber\\
\!\!\!\!\!\!\!\!\!\!\!\!\!\!\!&+&\!\!\!\!\!\G^g_d\GG_{0,-1}\!\otimes\!\DD+\G^d_g\DD\!\otimes\!\GG_{0,-1}
\!+\!\G^g_{d_{1}}\GG_{0,-1}\!\otimes\!\DD_{1}\!+\!\G^{d_{1}}_g\DD_{1}\!\otimes\!\GG_{0,-1}
\!+\!\G^g_{d_{2}}\GG_{0,-1}\!\otimes\!\DD_{2}\!+\!\G^{d_{2}}_g\DD_{2}\!\otimes\!\GG_{0,-1}
\nonumber\\
\!\!\!\!\!\!\!\!\!\!\!\!\!\!\!&+&\!\!\!\!\!\G^h_d\HH_{0,-1}\!\otimes\!\DD\!\!+\!\G^d_h\DD\!\otimes\!\HH_{0,-1}
\!\!+\!\G^h_{d_{1}}\HH_{0,-1}\!\otimes\!\DD_{1}\!\!+\!\G^{d_{1}}_h\DD_{1}\!\otimes\!\HH_{0,-1}
\!\!+\!\G^h_{d_{2}}\HH_{0,-1}\!\otimes\!\DD_{2}\!\!+\!\G^{d_{2}}_h\DD_{2}\!\otimes\!\HH_{0,-1}.
\end{eqnarray}
 Applying $\dir_{\bz}$ to $[\GG_{0,1},\GG_{0,-1}]=0$, by Lemma \ref{lemm-L+}, we obtain
\begin{eqnarray*}
&&\!\!\!\!\!\!\!\!\G^g_{d_2}=\G^h_{d_2}=\G^{d_2}_g=\G^{d_2}_h
=q^m\G^{ef}_{m,n-1}-\G^{ef}_{m,n}=q^{-m}\G^{fe}_{m,n}-\G^{fe}_{m,n-1}=0,\\
&&\!\!\!\!\!\!\!\!(q^m-1)\G^{gg}_{m,n-1}\!+\!(q^{-m}-1)\G^{gg}_{m,n}
\!=\!(q^m-1)\G^{gh}_{m,n-1}\!=\!(q^{-m}-1)\G^{hg}_{m,n}=0.
\end{eqnarray*}
Then we can deduce $\G^{ef}_{m,n}=\G^{fe}_{m,n}=0$ for $m,n\in\Z$ and $\G^{gg}_{m,n}=\G^{gh}_{m,n}=\G^{hg}_{m,n}=0$ for $m\neq0$.

Applying $\dir_{\bz}$ to $[\GG_{1,1},\GG_{0,-1}]=(1-q^{-1})\GG_{1,0}$ and using Lemma \ref{lemm-L+}, we get the identities
\begin{eqnarray*}
&&(1-q^{n-1})\G^{gg}_{0,n-1}=(1-q^n)\G^{g,h}_{0,n}= (1-q^{-n-1})\G^{hg}_{0,n}=0,\\
&&(1-q^{-n-1})\G^{gg}_{0,n}=\G^g_d=\G^g_{d_1}=\G^h_{d_1}=\G^d_g=\G^{d_1}_g=\G^{d_1}_h=0,
\end{eqnarray*}
which force $\G^{gg}_{0,n}=\G^{gh}_{0,n}=\G^{hg}_{0,n}=0$ for $n\neq0,-1$. Then we can rewrite $\dir_{\bz}(\GG_{0,-1})$ as
\begin{eqnarray}
\dir_{\bz}(\GG_{0,-1})=\!\!\!\!\mbox{$\sum\limits_{(m,n)\in\BZ^*\setminus(0,-1)}$}\!
\G^{hh}_{m,n}\HH_{m,n}\otimes\HH_{-m,-1-n}
+\G^h_d\HH_{0,-1}\otimes\DD+\G^d_h\DD\otimes\HH_{0,-1}.\label{equa-g-1}
\end{eqnarray}
Applying $\dir_{\bz}$ to $[\HH_{1,1},\GG_{0,-1}]=0$ and using \eqref{equa-g-1}, we have
\begin{eqnarray}
&&(1-q^{-1})\G^h_d=(1-q^{-1})\G^d_h=0,\nonumber\\ &&(q^m-q^n)\G^{hh}_{m-1,n-1}+(q^{-m}-q^{-1-n})\G^{hh}_{m,n}=0.\label{equa-G-1}
\end{eqnarray}
Then $\G^h_d=\G^d_h=0$. Take $m=n$ and $m=1+n$ respectively in \eqref{equa-G-1}, one has $\G^{hh}_{m,m}=0$ and $\G^{hh}_{m,m-1}=0$ for $m\in\Z^*$.
If $m\neq n,\,1+n$, then $\G^{hh}_{m-1,n-1}=\frac{q^{-1-n}-q^{-m}}{q^m-q^n}\,\G^{hh}_{m,n}$, which forces $\G^{hh}_{m,n}=0$ because of the fact that there are only finite many nonzero $\G^{hh}_{m,n}$'s. Hence $\dir_{\bz}(\GG_{0,-1})=0$.
Using similar arguments, we also obtain $\dir_{\bz}(\HH_{-1,0})=0$. By the similar method, we also have $\dir_{\bz}({\GG_{0,-n}})=\dir_{\bz}(\HH_{-n,0})=0$ for $n\in\Z^+$. Using Lemma \ref{lemm-L0} and the fact that $\WL_{m,n}$ for $m,n\in\Z^-$ can be generated by the set $\{\DD,\,\EE_{0,0},\,\FF_{0,0},\,\DD_{1},\,\DD_{2},\,\GG_{0,-n},\,\HH_{-n,0}\}$, we obtain the lemma.
\end{proof}

\begin{lemm}\label{lemm-D}
For any $\dir\in{\rm Der}(\WL,\VV)$, \eqref{summable} is a finite sum.
\end{lemm}
\begin{proof} Since $\dir=\mbox{$\sum_{\bk\in\BZ}\dir_\bk$},$ by the Lemmas \ref{lemm-L0}, \ref{lemm-L+} and \ref{lemm-L-}, one can suppose $\dir_{\bm}=(v_{\bm})_{\rm inn}$ for some $v_{\bm}\in\VV_{\bm}$ and $\bm\in\Z^2$. If $\Gamma'=\{\bm\in\BZ^*\,|\,v_{\bm}\neq0\}$ is an infinite set, by linear algebra, there exists $\rho\in\texttt{T}$
(cf.~notation in the proof of Lemma \ref{lemm-Dn}) such that $\rho(\bm)\neq0$ for all $\bm\in \Gamma'.$ Then $\dir(\rho)=\sum_{\bm\in\Gamma'}\rho(\bm)v_{\bm}$ is an infinite sum, which is not an element in $\VV.$ This contradicts to  the fact that $\dir\in{\rm Der}(\WL,\VV)$. The result follows.
\end{proof}

By now the proof of Proposition \ref{proposition} is completed.

\begin{lemm}\label{lemma5v}
Suppose $v\in\VV$ such that
$x\cdot v\in {\rm Im}(1-\tau)$ for all $x\in\WL.$ Then $v\in {\rm Im}(1-\tau)$.
\end{lemm}
\begin{proof} First note that $\WL\cdot {\rm Im}(1-\tau)\subset{\rm Im}(1-\tau).$ We shall prove that after a number of steps, by
replacing $v$ by $v - u$ for some $u\in {\rm Im}(1-\tau),$ the zero element is obtained and thus proving that $v\in{\rm Im}(1-\tau).$ Write $v=\sum_{{\bm\in\Z^2}}v_{\bm}.$ Obviously,
\begin{eqnarray}\label{eqrx}
v\in {\rm Im}(1-\tau)\ \,\Longleftrightarrow \ \,v_{\bm}\in {\rm Im}(1-\tau),\ \ \forall\,\,{\bm\in\BZ}.
\end{eqnarray}
For any $\bm\in\BZ^*$, choose $\rho\in\texttt{T}$ such that $\rho(\bm)\neq0.$ Then $$\rho\cdot v=\mbox{$\sum\limits_{\bm\in\BZ}$}\rho(\bm)v_{\bm}\in{\rm Im}(1-\tau),$$
which together with \eqref{eqrx} gives $v_{\bm}\in{\rm Im}(1-\tau)$. Thus by replacing $v$ by $v-\sum_{\bm\neq\bz}v_{\bm}$, we can suppose $v=v_{\bz}\in\VV_
{\bz}.$ Hence $v$ can be rewritten as  (cf.~statement after \eqref{first-})
{\small\begin{eqnarray*}
&&\ \mbox{$\sum\limits_{\bm}$}\big(\v^{ee}_{\bm}\EE_{\bm}\!\otimes\!\EE_{-\bm}
\!+\!\v^{ef}_{\bm}\EE_{\bm}\!\otimes\!\FF_{-\bm}\!+\!\v^{fe}_{\bm}\FF_{\bm}\!\otimes\!\EE_{-\bm}\!
+\!\v^{ff}_{\bm}\FF_{\bm}\!\otimes\!\FF_{-\bm}\big)
\!+\!\mbox{$\sum\limits_{\bk}$}\big(\v^{eg}_{\bk}\EE_{\bk}\!\otimes\!\GG_{-\bk}\\[-6pt]
&&+\v^{eh}_{\bk}\EE_{\bk}\!\otimes\!\HH_{-\bk}\!\!+\!\v^{he}_{\bk}\HH_{\bk}\!\otimes\!\EE_{-\bk}\!\!
+\!\v^{ge}_{\bk}\GG_{\bk}\!\otimes\!\EE_{-\bk}\!\!+\!\v^{fg}_{\bk}\FF_{\bk}\!\otimes\!\GG_{-\bk}
\!+\!\v^{fh}_{\bk}\FF_{\bk}\!\otimes\!\HH_{-\bk}\!+\!\v^{gf}_{\bk}\GG_{\bk}\!\otimes\!\FF_{-\bk}\\
&&+\v^{hf}_{\bk}\HH_{\bk}\!\otimes\!\FF_{-\bk}\!+\!\v^{gg}_{\bk}\GG_{\bk}\!\otimes\!\GG_{-\bk}
\!+\!\v^{gh}_{\bk}\GG_{\bk}\!\otimes\!\HH_{-\bk}+\v^{hg}_{\bk}\HH_{\bk}\!\otimes\!\GG_{-\bk}\!
+\!\v^{hh}_{\bk}\HH_{\bk}\otimes\HH_{-\bk}\big)\\
&&+\v^e_d\EE_{\bz}\otimes\DD\!+\!\v^e_{d_{1}}\EE_{\bz}\otimes\DD_{1}
\!+\!\v^e_{d_{2}}\EE_{\bz}\otimes\DD_{2}\!+\!\v^f_d\FF_{\bz}\otimes\DD
\!+\!\v^f_{d_{1}}\FF_{\bz}\otimes\DD_{1}\!+\!\v^f_{d_{2}}\FF_{\bz}\!\otimes\!\DD_{2}
\!+\v^d_e\DD\!\otimes\!\EE_{\bz}\\
&&+\v^d_f\DD\otimes\FF_{\bz}+\v^d_d\DD\!\otimes\!\DD+\v^d_{d_{1}}\DD\!\otimes\!\DD_{1}
+\v^d_{d_{2}}\DD\otimes\DD_{2}+\v^{d_{1}}_e\DD_{1}\!\otimes\!\EE_{\bz}\!+\!\v^{d_{1}}_f\DD_{1}\!\otimes\!\FF_{\bz}\!
+\!\v^{d_{1}}_d\DD_{1}\!\otimes\!\DD\\
&&+\v^{d_{1}}_{d_{1}}\DD_{1}\!\otimes\!\DD_{1}\!\!+\!\!\v^{d_{1}}_{d_{2}}\DD_{1}\!\otimes\!\DD_{2}\!
+\!\v^{d_{2}}_e\DD_{2}\!\otimes\!\EE_{\bz}\!+\!\v^{d_{2}}_f\DD_{2}\!\otimes\!\FF_{\bz}\!
+\!\v^{d_{2}}_d\DD_{2}\!\otimes\!\DD\!+\!\v^{d_{2}}_{d_{1}}\DD_{2}\!\otimes\!\DD_{1}
\!+\!\v^{d_{2}}_{d_{2}}\DD_{2}\!\otimes\!\DD_{2}.
\end{eqnarray*}\small}%
\noindent 
The fact that $\DD\cdot v\in{\rm Im}(1-\tau)$ implies
\begin{eqnarray*}
&&\v^{e,e}_{\bm}+\v^{e,e}_{-\bm}=\v^{f,f}_{\bm}+\v^{f,f}_{-\bm}= \v^{e,g}_{\bk}+\v^{g,e}_{-\bk}=0,\\
&&\v^{e,h}_{\bk}+\v^{h,e}_{-\bk}=\v^{f,g}_{\bk}+\v^{g,f}_{-\bk}= \v^{f,h}_{\bk}+\v^{h,f}_{-\bk}=\v^e_d+\v^d_e=0,\\
&&\v^e_{d_1}+\v^{d_1}_e=\v^e_{d_2}+\v^{d_2}_e=
\v^f_d\!+\!\v^d_f=\v^f_{d_1}+\v^{d_1}_f=\v^f_{d_2}+\v^{d_2}_f=0.
\end{eqnarray*}
Replacing $v$ by $v-u_1$, where $u_1$ is equal to
{\small
\begin{eqnarray*}
&&\mbox{$\sum\limits_{\bm}$}\big(\v^{ee}_{\bm}\EE\!_{\bm}\!\otimes\!\EE\!_{-\bm}
\!\!+\!\v^{ff}_{\bm}\FF\!_{\bm}\!\otimes\!\FF\!_{-\bm}
\!\!+\!\v^{fe}_{\bm}(\FF\!_{\bm}\!\otimes\!\EE\!_{-\bm}\!\!-\!\EE\!_{-\bm}\!\otimes\!\FF\!_{\bm})\big)
\!+\!\mbox{$\sum\limits_{\bk}$}\big(\v^{eg}_{\bk}(\EE\!_{\bk}\!\otimes\!\GG\!_{-\bk}\!\!-\!\GG\!_{-\bk}\!\otimes\!\EE\!_{\bk})\\[-6pt]
&&+\v^{eh}_{\bk}(\EE_{\bk}\!\otimes\!\HH_{-\bk}\!-\!\HH_{-\bk}\!\otimes\!\EE_{\bk})\!\!
+\!\v^{fg}_{\bk}(\FF_{\bk}\!\otimes\!\GG_{-\bk}\!-\!\GG_{-\bk}\!\otimes\!\FF_{\bk})
\!+\!\v^{fh}_{\bk}(\FF_{\bk}\!\otimes\!\HH_{-\bk}\!-\!\HH_{-\bk}\!\otimes\!\FF_{\bk})\\
&&+\!\v^{hg}_{\bk}(\HH_{\bk}\!\otimes\!\GG\!_{-\bk}\!-\!\GG\!_{-\bk}\!\otimes\!\HH_{\bk})\big)
\!+\!\v^e_d(\EE_{\bz}\!\otimes\!\DD\!\!-\!\DD\!\otimes\!\EE_{\bz})\!+\!v^e_{d_{1}}(\EE_{\bz}\!\otimes\!\DD_{1}\!\!-\!\DD_{1}\!\otimes\!\EE_{\bz})
\!+\!\v^e_{d_{2}}(\EE_{\bz}\!\otimes\!\DD_{2}\!\!-\!\DD_{2}\!\otimes\!\EE_{\bz})\\
&&+\!\v^f_d(\FF_{\bz}\otimes\DD-\DD\otimes\FF_{\bz})
\!+\!\v^f_{d_{1}}(\FF_{\bz}\otimes\DD_{1}\!-\DD_{1}\!\otimes\!\FF_{\bz})+
\!\v^f_{d_{2}}(\FF_{\bz}\!\otimes\!\DD_{2}-\DD_{2}\!\otimes\!\FF_{\bz})\ \in\ {\rm Im}(1-\tau),
\end{eqnarray*}
\small}
we can simplify $v$  as
{\small\begin{eqnarray*}
&&\mbox{$\sum\limits_{\bm}$}\big(\v^{ef}_{\bm}\EE_{\bm}\!\otimes\!\FF_{-\bm}\big)
+\mbox{$\sum\limits_{\bk}$}\big(\v^{gg}_{\bk}\GG_{\bk}\!\otimes\!\GG_{-\bk}
\!+\!\v^{hh}_{\bk}\HH_{\bk}\otimes\HH_{-\bk}\big)
+\mbox{$\sum\limits_{\bk}$}\big(\!\v^{gh}_{\bk}\GG_{\bk}\!\otimes\!\HH_{-\bk}\big)\\[-6pt]
&&+\v^d_d\DD\otimes\DD+\v^d_{d_{1}}\DD\otimes\DD_{1}
+\v^d_{d_{2}}\DD\otimes\DD_{2}+\v^{d_{1}}_d\DD_{1}\otimes\DD
+\v^{d_{1}}_{d_{1}}\DD_{1}\otimes\!\DD_{1}+\v^{d_{1}}_{d_{2}}\DD_{1}\otimes\DD_{2}\\
&&+\v^{d_{2}}_d\DD_{2}\otimes\DD+\v^{d_{2}}_{d_{1}}\DD_{2}\otimes\DD_{1}
+\v^{d_{2}}_{d_{2}}\DD_{2}\otimes\DD_{2},
\end{eqnarray*}
\small}
\noindent where the coefficients are all in $\C[q^{\pm1}]$. The fact that $\GG_{0,1}\cdot v\in{\rm Im}(1-\tau)$ forces
\begin{eqnarray*}
q^{m_1}\v^{ef}_{m_1,m_2-1}-\v^{ef}_{m_1,m_2}=(q^{m_1}-1)\v^{gh}_{m_1,m_2-1}=
\v^{d_2}_{d_2}=\v^{d}_{d_2}+\v^{d_2}_d=\v^{d_2}_{d_1}+\v^{d_1}_{d_2}=0,
\end{eqnarray*}
which in particular gives $\v^{ef}_{m,n}=0$. Replacing $v$ by $v-u_1$, where $u_1=\v^d_{d_2}(\DD\otimes\DD_2-\DD_2\otimes\DD)+\v^{d_2}_{d_1}(\DD_2\otimes\DD_1-\DD_1\otimes\DD_2)$, we have $\v^d_{d_2}=\v^{d_2}_d=\v^{d_2}_{d_1}=\v^{d_1}_{d_2}=0$.
The fact that $\GG_{1,0}\cdot v\in{\rm Im}(1-\tau)$ forces
$
(q^{m_2}-1)\v^{g,h}_{m_1-1,m_2}=\v^{d_1}_{d_1}=\v^{d}_{d_1}+\v^{d_1}_d=0,
$ 
which gives $\v^{g,h}_{m,n}=0$ for $n\in\Z^*$. Replacing $v$ by $v-u$, where $u=\v^d_{d_1}(\DD\otimes\DD_1-\DD_1\otimes\DD)$, one can rewrite $v$ as
$
v=\mbox{$\sum_{\bk}$}\big(\v^{gg}_{\bk}\GG_{\bk}\otimes\GG_{-\bk}
+\v^{hh}_{\bk}\HH_{\bk}\otimes\HH_{-\bk}\big)+\v^d_d\DD\otimes\DD.
$ 
The fact that $\EE_{\bz}\cdot v\in{\rm Im}(1-\tau)$ forces
$
\v^d_d=\v^{gg}_{\bk}+\v^{gg}_{-\bk}=\v^{hh}_{\bk}+\v^{hh}_{-\bk}=0.
$ 
Hence $v\in{\rm Im}(1-\tau).$
\end{proof}

 \ni{\it Proof of Theorem \ref{main}.}\ \ Let $(\WL
,[\cdot,\cdot],\dir)$ be a Lie bialgebra structure on $\WL$. By
(\ref{bLie-d}), (\ref{deriv}) and Proposition \ref{proposition},
$\D=\D_r$ is defined by (\ref{D-r}) for some $r\in\WL\otimes\WL.$ By
(\ref{cLie-s-s}), ${\rm Im}\,\D\subset{\rm Im}(1-\tau).$ By
Lemma \ref{lemma5v}, $r\in{\rm Im}(1-\tau).$ Then (\ref{cLie-s-s}),
Lemma \ref{some}(2) and Corollary \ref{coro1} show that $c(r)=0.$ Thus $(\WL ,[\cdot,\cdot],\D)$ is a
triangular coboundary Lie bialgebra.\QED\vskip12pt

\noindent{\bf Acknowledgements}\ \ The authors would sincerely like to thank Professor Yucai Su for his supervision and invaluable comments.
\end{CJK*}

\vskip8pt


\footnotesize
\begin{thebibliography}{9999}\vskip0pt\small
\parindent=2ex\parskip=-1pt\baselineskip=-1pt
\small

\bibitem{AG} B.  Allison, Y. Gao, The root system and the core of an extended
affine Lie algebra, {\it Selecta Math. (N.S.)} {\bf 7}(2) (2001), 149--212.

\bibitem{AABGP} B. Allison, S. Azam, S. Berman, Y. Gao, A. Pianzola, Extended affine Lie algebras and their root systems, {\it Mem. Amer. Math. Soc.} {\bf 126}(603) (1997).

\bibitem{B} Y. Billig, Representations of toroidal extended affine Lie algebras, {\it J. Algebra} {\bf 308} (2007), 252--269.

\bibitem{BGK} S. Berman, Y. Gao, Y. Krylyuk, Quantum tori and the structure of elliptic quasi-simple Lie algebras, {\it J. Funct. Anal.} {\bf 135} (1996) 339--389.

\bibitem{BGKE} S. Berman, Y. Gao, Y. Krylyuk,  E. Neher, The alternative torus and the structure of elliptic quasisimple Lie algebras of type $A_2$, {\it Trans. Amer. Math. Soc.} {\bf 347}(11) (1995), 4315--4363.

\bibitem{BL} Y. Billig, M. Lau, Irreducible modules for extended affine Lie algebras, {\it J. Algebra} {\bf 327} (2011) 208--235.

\bibitem{CS} Y. Cheng, Y. Shi, Lie bialgebra structures on the $q$-analog Virasoro-like algebras, {\it Comm. Algebra} {\bf 37}(4) (2009), 1264--1274.

\bibitem{D1} V. Drinfeld, Constant quasiclassical solutions of the
Yang-Baxter quantum equation, {\it Soviet Math. Dokl.} {\bf28}(3) (1983),
667--671.

\bibitem{D2} V. Drinfeld, Quantum groups, in: {\it Proceeding of the
International Congress of Mathematicians}, Vol. 1, 2, Berkeley,
Calif. 1986, Amer. Math. Soc., Providence, RI, 1987, 798--820.

\bibitem{G1} Y. Gao, Representations of extended affine Lie algebras coordinatized by certain quantum tori, {\it Compositio Math.} {\bf 123}(1) (2000), 1--25\vs{-7pt}.

\bibitem{G2} Y. Gao, Fermionic and bosonic representations of the extended affine Lie algebra \mbox{\footnotesize$\widetilde{\frak{gl}_N(\mathbb{C}_q)}$}, {\it Cananian Math. Bull.} {\bf 45} (2002), 623--633\vs{-7pt}.

\bibitem{GZ} Y. Gao, Z. Zeng, Hermitian representations of the extended affine Lie algebra \mbox{\footnotesize$\widetilde{\frak{gl}_2(\mathbb{C}_q)}$}, {\it Adv. Math.} {\bf 207}(1) (2006), 244--265.

\bibitem{HLS} J. Han, J. Li, Y. Su, Lie bialgebra structures on the
    Schr\"odinger-Virasoro Lie algebra, {\it J. Math. Phys} {\bf 50} (2009), 083504.

\bibitem{HT} R. H{\o}egh-Krohn, B. Torresani, Classification and construction of quasi-simple Lie algebras, {\it J. Funct. Anal.} {\bf 89} (1990), 106--136.

\bibitem{JM} C. Jiang, D. Meng, The derivation algebra of the associative algebra $\C_q[X,Y,X^{-1},Y^{-1}]$, {\it Comm. Algebra} {\bf 26}(6) (1998), 1723--1736.

\bibitem{LS} W. Lin, Y. Su, Modules for the core of extended affine Lie algebras of type $A_1$ with coordinates in rank 2 quantum tori, {\it Pacific J. Math.} {\bf 242}(1) (2009), 143--166.

\bibitem{LSX} J. Li, Y. Su, B. Xin, Lie bialgebras of a family of Block
type, {\it Chinese Annals of Math. B} {\bf 29} (2008), 487--500.

\bibitem{M} K. Mini, Integrable irreducible highest weight modules for $sl_2(\C_p[x^{\pm 1},y^{\pm 1}])$, {\it Osaka J. Math.} {\bf 41} (2004), 295--326.

\bibitem{Mi} W. Michaelis, A class of infinite-dimensional Lie
bialgebras containing the Virasoro algebras, {\it Adv. Math.}
{\bf 107} (1994), 365--392.

\bibitem{MY} Y.I. Manin, Topics in noncommutative geometry, Princeton Univ. Press, 1991.

\bibitem{N} E. Neher, Extended affine Lie algebras, {\it C. R. Math. Acad. Sci. Soc. R. Can.} {\bf 26} (2004) 90--96.

\bibitem{NT} S.H. Ng, E.J. Taft, Classification of the Lie bialgebra
structures on the Witt and Virasoro algebras, {\it J. Pure Appl.
Alg.} {\bf 151} (2000), 67--88.

\bibitem{SS} G. Song, Y. Su, Lie bialgebras of
generalized Witt type, {\it Science in China A}
{\bf 49}(4) (2006), 533--544.

\bibitem{SS1}	G. Song, Y. Su, B. Xin,
Quantization of Hamiltonian-type Lie algebras, {\it Pacific J. Math.} {\bf240} (2009), 371--381.

\bibitem{SZ} Y. Su, K. Zhao, Generalized Virasoro and super-Virasoro algebras and modules of the intermediate series, {\it J. Algebra} {\bf 252} (2002), 1--19.

\bibitem{T} E.J. Taft, Witt and Virasoro algebras as Lie bialgebras, {\it
J. Pure Appl. Algebra} {\bf 87} (1993), 301--312.

\bibitem{WSS} Y. Wu, G. Song, Y. Su, Lie bialgebras of generalized
Virasoro-like type, {\it Acta Mathematica Sinica, English Series}
{\bf 22} (2006), 1915--1922.

\bibitem{WSS1} Y. Wu, G. Song, Y. Su, Lie bialgebras of generalized Witt type. II.,
{\it Comm. Algebra} {\bf35} (2007), 1992--2007.






\bibitem{YS}X. Yue, Y. Su, Lie bialgebra structures on Lie algebras of
generalized Weyl type, {\it Comm. Algebra} {\bf36} (2008), 1537--1549.


\end{thebibliography}
\end{document}